\documentclass[preprint,12pt]{elsarticle}

\usepackage{amssymb}
\usepackage{amsmath}
\usepackage{amsthm}
\usepackage{xcolor}
\usepackage{dsfont}

\usepackage[
    left=0.5in,
    right=0.5in,
    top=1in,
    bottom=1in
]{geometry}
\usepackage{subfig}
\usepackage{float}
\newtheorem{rem}{Remark}[section]
\newtheorem{definition}{Definition}[section]
\newtheorem{thm}{Theorem}[section]
\newtheorem{lem}{Lemma}[section]

\newtheorem{prop}{Proposition}[section]

\journal{Journal of Mathematical Analysis and Applications}

\begin{document}

\begin{frontmatter}



\title{Asymptotic Analysis of Discrete-Time Hawkes Process} 


\author{Utpal Jyoti Deba Sarma} 
\author{Dharmaraja Selvamuthu}

\affiliation{
            addressline={Department of Mathematics, Indian Institute of Technology Delhi}, 
            city={New Delhi},
            postcode={110016}, 
            state={Delhi},
            country={India}}

\begin{abstract}
In a discrete-time setting, we consider an arrival process $\left\{\xi_n \, \middle| \, n = 1, 2, \ldots \right\}$, which models the occurrence of events, and a corresponding point process $\left\{H_n \, \middle| \, n = 1, 2, \ldots \right\}$, known as the discrete-time Hawkes process. These two stochastic processes are related by $H_n = \sum_{i=1}^n \xi_i$, and exhibit a self-exciting property. In particular, we study the limiting behavior of the arrival process and establish the Large Deviation Principle for the discrete-time Hawkes process.  We also illustrate an application in which insurance claims are modeled using the discrete-time Hawkes process and analyze
its behavior. 
\end{abstract}




\end{frontmatter}



\section{Introduction}
The continuous-time Hawkes process is a class of point process used to model events whose occurrence depends on their past history. It was introduced by Alan G. Hawkes \cite{hawkes1974cluster, hawkes1971spectra}, and due to its self-exciting and clustering properties, it has found many applications in areas like finance, seismology, etc. (see Bacry et al. \cite{bacry2013some} and Vere-Jones \cite{vere1975stochastic}). Although the continuous-time Hawkes process has extensive literature, it may not always be applicable for certain contexts where data are recorded in a discrete-time framework or in an aggregated form. In such cases, the discrete-time Hawkes process provides a more efficient alternative. Despite its growing importance, the discrete-time Hawkes process is relatively new and many of its properties are yet to be studied.

This paper studies the discrete-time Hawkes process introduced by Seol \cite{seol2015limit}. The discrete-time Hawkes process is a class of $g$-functions; for a note on $g$-functions, refer to Berbee \cite{berbee1987chains}. In this model, a sequence of $\left\{0,1 \right\}$-valued Bernoulli random variables  $\left\{\xi_n \, \middle| \, n=1,2,\ldots \right\}$ is considered, which is called the arrival process. The arrival process controls the arrivals of events, and the occurrence probabilities of these events depend on their past history. The discrete-time Hawkes process $\left\{H_n \, \middle| \, n=1,2,\ldots \right\}$ is defined as the number of events that occurred up to time $n$ following the arrival process $\left\{\xi_n \, \middle| \, n=1,2,\ldots \right\}$. Specifically, an arrival occurs at time $n$ if $\xi_n = 1$, while $\xi_n = 0$ indicates no arrival. 

Substantial progress has been made in the study of both linear and nonlinear continuous-time Hawkes process. The Law of Large Numbers for the linear continuous-time Hawkes process is proved by Daley and Vere-Jones  \cite{daley2003introduction}, and the Central Limit Theorem (CLT) is obtained by Bacry et al. \cite{bacry2013some}. For the nonlinear continuous-time Hawkes process, the Law of Large Numbers is obtained using the Ergodic Theorem by Br\'emaud and Massouli\'e \cite{bremaud1996stability}, and the CLT is proved by Zhu \cite{zhu2013central}. The Large Deviation Principle (LDP) for the linear continuous-time Hawkes process is established by Bordenave and Torrisi \cite{bordenave2007large}, and for the nonlinear case, it is proved by Zhu \cite{zhu2014process}. However, comparatively not much is known about the discrete-time Hawkes process. Previous studies have established the Weak Law of Large Numbers (WLLN), the Strong Law of Large Numbers (SLLN), and the CLT for the discrete-time Hawkes process in Sarma and Selvamuthu \cite{sarma2024study}, and Seol \cite{seol2015limit}. Based on this framework, we study the asymptotic behavior of the arrival process $\left\{\xi_n \, \middle| \, n=1,2,\ldots \right\}$, and prove the LDP for the discrete-time Hawkes process $\left\{H_n \, \middle| \, n=1,2,\ldots \right\}$.

In a related work, Wang \cite{wang2023large} established the LDP for a discrete-time marked Hawkes process, in which Poisson random variables and some i.i.d. random variables were combined to generate the model. In contrast, the model studied in this paper directly generates the discrete-time Hawkes process using the arrival process, based on probabilities that are dependent on the past history. Recently, a discrete-time self-exciting model was applied to study the infection and death from COVID-19 (see Browning et al. \cite{browning2021simple}). From here onward, we refer to the discrete-time Hawkes process as ``DTHP".

This paper is organized as follows. In Section \ref{sec:2}, we present the description of the DTHP and review the existing literature on the DTHP. We also provide some preliminaries related to the Large Deviation Principle, associated random variables, and nearly subadditive sequences. Section \ref{sec:3} focuses on the asymptotic behavior of the arrival process $\left\{\xi_n \, \middle| \, n=1,2,\ldots \right\}$. In Section \ref{sec:4}, we prove the LDP for the DTHP $\left\{H_n \, \middle| \, n=1,2,\ldots \right\}$. We also study the convergence of scaled logarithmic moment generating functions (scaled logarithmic MGF) of the random variables $H_n, n = 1,2,\ldots$, and the bounds for the limit function. In Section \ref{Sec 5}, we provide an illustration in which insurance claims are modeled using the DTHP and analyze its behavior. Finally, in Section \ref{Sec 6.1}, we give a summary of the results obtained in this paper and possible extensions of this work in future.

\section{Preliminaries} \label{sec:2}

Throughout the paper, $\left(\Omega,\mathcal{F},\left\{\mathcal{F}_n \right\}_{n=1}^\infty,\mathbb{P}\right)$ is the underlying filtered probability space in which the random variables are defined. The mathematical description of the DTHP is given in the following subsection.

\subsection{The Model}\label{sec model}\label{Self-exciting process}

Let $\left(a_i\right)_{i=0}^{\infty}$ be a sequence of positive real numbers with the following assumptions,
\begin{equation*}
\sum_{i=0}^{\infty} a_i < 1
\qquad \mbox{and} \qquad 
\sum_{i=1}^{\infty} ia_i < \infty.
\end{equation*}
Let $\left\{\xi_n \, \middle| \, n=1,2,\ldots \right\}$ be defined as a sequence of random variables taking values in $\left\{0,1\right\}$ such that
\begin{equation*}
\mathbb{P}\left(\xi_1=1\right)=a_0 
\qquad \text{and} \qquad
\mathbb{P}\left(\xi_1=0\right)=1-a_0,
\end{equation*}
and for $n\geq 2$
\begin{equation*}
\mathbb{P}\left(\xi_n=1 \, \middle| \, \xi_1,\ldots,\xi_{n-1}\right)=a_0+\sum_{i=1}^{n-1}a_{n-i}\xi_i,
\end{equation*}
and
\begin{equation*}
\mathbb{P}\left(\xi_n=0 \, \middle| \, \xi_1,\ldots,\xi_{n-1}\right)=1-\left(a_0+\sum_{i=1}^{n-1}a_{n-i}\xi_i\right).
\end{equation*}

The arrival of events occurs w.r.t. the stochastic process $\left\{\xi_n \, \middle| \, n=1,2,\ldots \right\}$, which is called the arrival process.
The focus of this paper is to study the limiting behavior of the arrival process $\left\{\xi_n \, \middle| \, n=1,2,\ldots \right\}$, and the DTHP, which is denoted by $\left\{H_n \, \middle| \, n=1,2,\ldots \right\}$, and defined as
\begin{equation*}
    H_n=\sum_{i=1}^n\xi_i.
\end{equation*}
The sequence $\left(a_i\right)_{i=0}^\infty$ is called the exciting function of the DTHP,
and 
\begin{equation*}
    \lambda_n=a_0 + \sum_{i=1}^{n-1} a_{n-i}\xi_i ,
\end{equation*}
is called the intensity of the DTHP. 
 \begin{rem}
      Each arrival at time $n$ increases the intensity of the DTHP at that time instant, establishing its self-exciting nature, and the intensity depends on its entire history.
 \end{rem}
\begin{rem}
     The filtration $\left\{\mathcal{F}_n \, \middle| \,n=1,2,\ldots \right\}$ is defined as the natural filtration generated by the arrival process $\left\{\xi_n \, \middle| \, n=1,2,\ldots \right\}$, i.e., $\mathcal{F}_n=\sigma\left(\xi_1,\ldots,\xi_n\right)$.
\end{rem}

\subsection{Limiting Behavior of the DTHP}
The WLLN and CLT for the DTHP have been proven in Seol \cite{seol2015limit}, i.e.,
\begin{equation*} 
   \frac{\sum_{i=1}^n \xi_i}{n}=\frac{H_n}{n} \to \frac{a_0}{1-\sum_{i=1}^{\infty} a_i}, \qquad \text{as} ~ n \to \infty ~ \text{in probability},
\end{equation*}
and,
\begin{equation*} 
\frac{H_n-n \frac{a_0}{1-\sum_{i=1}^{\infty} a_i}}{\sqrt{n}} \to \mathcal{N}\left(0, \frac{\left(1-\frac{a_0}{1-\sum_{i=1}^{\infty} a_i}\right)\frac{a_0}{1-\sum_{i=1}^{\infty} a_i}}{\left(1-\sum_{j=1}^{\infty} a_j\right)^2}\right), \qquad \text{as} ~ n \to \infty ~ \text{in distribution}.
\end{equation*}
Further, the SLLN for the DTHP is obtained in Sarma and Selvamuthu \cite{sarma2024study}, i.e.,
\begin{equation}\label{con:2.3}
   \frac{\sum_{i=1}^n \xi_i}{n}=\frac{H_n}{n} \to \frac{a_0}{1-\sum_{i=1}^{\infty} a_i}, \qquad \text{as} ~ n \to \infty ~ \text{almost surely (a.s.)}.
\end{equation}
Sarma and Selvamuthu \cite{sarma2024study} also analyzed the compensator of the DTHP and obtained its WLLN, SLLN, and CLT.

\subsection{Large Deviation Theory}
Let $\mathbb{X}$ be a topological subspace of $\mathbb{R}$ (with the usual topology generated by open subsets of $\mathbb{R}$), and $\left\{X_n \, \middle| \, n=1,2,\ldots \right\}$ be a sequence of $\mathbb{X}$-valued random variables. Define 
\begin{equation*}
S_n=\sum_{i=1}^n X_i.
\end{equation*}
\begin{definition}[Rate function]
     (Dembo and Zeitouni \cite{dembo2009large}) A function $I \colon \mathbb{X} \to [0,\infty]$ is called a rate function if it is lower semi-continuous, i.e., for each $\theta \in [0, \infty)$, the level sets $\left\{x \, \middle| \, I(x) \leq \theta \right\}$ are closed subsets of $\mathbb{X}$. Moreover, if the level sets are compact subsets of $\mathbb{X}$, then $I$ is called a good rate function.
\end{definition}
\begin{definition}[Large Deviation Principle]
     (Dembo and Zeitouni \cite{dembo2009large}) The sequence  $\left\{S_n \, \middle| \, n=1,2,\ldots \right\}$  is said to satisfy the Large Deviation Principle with rate function $I \colon \mathbb{X} \to [0,\infty]$ if 
\begin{enumerate}
    \item for any closed set $F \subseteq \mathbb{X}$,
    \begin{equation*}
        \limsup_{n \to \infty} \frac{1}{n}\log\mathbb{P}\left(\frac{S_n}{n} \in F \right) \leq -\inf_{x \in F} I(x),
    \end{equation*}
    \item  for any open set $G \subseteq \mathbb{X}$,
    \begin{equation*}
        \liminf_{n \to \infty} \frac{1}{n}\log\mathbb{P}\left(\frac{S_n}{n} \in G \right) \geq -\inf_{x \in G} I(x).
    \end{equation*}
\end{enumerate}
\end{definition}
The following theorem gives sufficient conditions for the existence of such a rate function $I.$ 

\begin{thm}[G\"artner-Ellis Theorem] 
Assume that, for every $t \in \mathbb{R}$,  the following limit exists, 
\begin{equation*}
    \Lambda(t)=\lim_{n\rightarrow \infty}\frac{1}{n}\log \mathbb{E}\left(e^{tS_n}\right).
\end{equation*}
Then, its Fenchel-Legendre transform (for more details on Fenchel-Legendre transform, see Hiriart-Urruty and Mart\'inez-Legaz \cite{hiriart2003new}), 
\begin{equation*}
    \Lambda^*(x)= \sup_{t \in \mathbb{R}} \left\{tx - \Lambda(t)\right\},
\end{equation*} 
verifies the conditions below,
\begin{enumerate}
    \item \label{GE 1} for every closed set $F\subseteq \mathbb{R}$,
        \begin{equation*}
            \limsup_{n\rightarrow \infty} \frac{1}{n} \log\mathbb{P}\left(\frac{S_n}{n}\in F \right) \leq -\inf_{x\in F} \Lambda^*(x),
        \end{equation*}
        \item if $\Lambda$ is differentiable, then for every open set  $G\subseteq\mathbb{R}$,
        \begin{equation*}
            \liminf_{n\rightarrow \infty} \frac{1}{n} \log\mathbb{P}\left(\frac{S_n}{n}\in G \right) \geq -\inf_{x\in G} \Lambda^*(x).
        \end{equation*}
\end{enumerate}  
\end{thm}
\begin{proof}
    For proof, see Theorem 2.3.6 in Dembo and Zeitouni \cite{dembo2009large}. 
\end{proof}    

The G\"artner-Ellis Theorem requires the function $\Lambda$ to be differentiable to conclude the existence of LDP, which is difficult to obtain in many cases. Thus, an alternate characterization for the existence of LDP is given in
Theorem \ref{th 2.3}. Before stating the theorem, we give a definition of exponentially tight measure/distribution.
\begin{definition}[Exponentially tight measure/distribution] (Dembo and Zeitouni \cite{dembo2009large})\label{def_ET}
     Suppose for each $n \in \mathbb{N}$, $\mathbb{P}_n$ is a probability measure/distribution on the measurable space $\left(\mathbb{X}, \mathcal{B}\left(\mathbb{X} \right)\right)$, where $\mathcal{B}\left(\mathbb{X} \right)$ denotes the set of all Borel subsets of $\mathbb{X}$. The sequence of probability measure/distribution $\left\{\mathbb{P}_n \, \middle| \, n=1,2,\ldots \right\}$ is said to be exponentially tight if, for any $\varepsilon < \infty,$ there exists a compact set $K_{\varepsilon} \subseteq \mathbb{X}$ such that
    \begin{equation*}
        \limsup_{n \to \infty}\frac{1}{n}\log \mathbb{P}_n \left(K^{\complement}_{\varepsilon} \right) < - \varepsilon.
    \end{equation*}
\end{definition}
\begin{thm}[Bryc's Theorem]\label{th 2.3}
 Assume that the distributions of $\frac{S_n}{n}$ are exponentially tight and that, for every continuous and bounded function $g$ defined on $\mathbb{X}$, the following limit exists,
 \begin{equation*}
 \Lambda_g = \lim_{n \to \infty} \frac{1}{n} \log \mathbb{E}\left(e^{ng\left(\frac{S_n}{n}\right)}\right).
 \end{equation*}
Then $\left\{S_n \, \middle| \, n=1,2,\ldots \right\}$ satisfies the LDP with a good rate function 
 \begin{equation*}
 I(x) = \sup\limits_{g \in \mathcal{C}_b\left(\mathbb{X}\right)}\left\{g(x) - \Lambda_g\right\},
 \end{equation*}
 where $\mathcal{C}_b\left(\mathbb{X}\right)$ is the family of bounded and continuous functions defined on $\mathbb{X}$. 
\end{thm}
\begin{proof}
    For proof, see Theorem 4.4.2 in Dembo and Zeitouni \cite{dembo2009large}.
\end{proof}
The requirement to verify the existence of the limit $\Lambda_g$ for the entire family of continuous and bounded functions can be relaxed by performing the verification on a suitably chosen, smaller collection of functions called a well-separating collection of functions.

\begin{definition}[Well-separating functions]\label{def 2.3}
    (Dembo and Zeitouni \cite{dembo2009large}) A collection $ \mathcal{G} $ of continuous real-valued functions defined on $\mathbb{X}$ is called well-separating if
\begin{itemize}
    \item[(a)] the constant functions are in $ \mathcal{G} $,
    \item[(b)] if $ g_1, g_2 \in \mathcal{G} $, then $ g(x) = \min(g_1(x), g_2(x)) \in \mathcal{G} $,
    \item[(c)] given $ x, y \in \mathbb{X}, x \neq y $, and $ a, b \in \mathbb{R} $, there exists $ g \in \mathcal{G} $ such that $ g(x) = a $ and $ g(y) = b $.
\end{itemize}
\end{definition}

The following result is helpful for checking the existence of LDP using a well-separating collection of functions.

\begin{thm}
 If, with the notation of Theorem \ref{th 2.3}, the limit $ \Lambda_g $ exists for a well-separating collection of functions defined on $\mathbb{X}$, then it exists for every continuous and bounded function defined on $\mathbb{X}$, that is, the assumption of the Bryc's Theorem is satisfied.
\end{thm}
\begin{proof}
    For proof, see Theorem 4.4.10 in Dembo and Zeitouni \cite{dembo2009large}.
\end{proof}

\subsection{Associated Random Variables and Nearly subadditive Sequences}

\begin{definition}[Associated random variables]
    (Oliveira \cite{oliveira2012asymptotics}) Let $f_1,f_2 \colon \mathbb{R}^n \to \mathbb{R}$ be two co-ordinatewise non-decreasing functions. The random variables $X_1,\ldots,X_n$ are said to be associated if \begin{equation*}
        \mathrm{Cov}\left[f_1\left(X_1,\ldots,X_n\right),f_2\left(X_1,\ldots,X_n \right)\right] \geq 0,
    \end{equation*}
    whenever the covariance exists.
\end{definition}

\begin{definition}[Nearly subadditive sequences]
(de Bruijn and Erd\H{o}s \cite{de1952some}) Let $\left(y_n\right)_{n=1}^{\infty}$ be a non-decreasing sequence of non-negative real numbers. A sequence of real numbers $\left(x_n\right)_{n=1}^{\infty}$ is called nearly subadditive if
\begin{equation*}
    x_{m+n} \leq x_m + x_n + y_{m+n}, \qquad \text{for all} ~ m,n \in \mathbb{N}.
\end{equation*}
The sequence $\left(y_n\right)_{n=1}^{\infty}$ is called the error term.
\end{definition}
The following is an important theorem that will be used in the proof of convergence of the scaled logarithmic MGF of the random variables $H_n, n=1,2,\ldots$.

\begin{thm}\label{th:2.4}
    Let $\left(x_n\right)_{n=1}^{\infty}$ be a nearly subadditive sequence with error term $\left(y_n\right)_{n=1}^{\infty}$. If
    \begin{equation*}
        \sum_{n=1}^{\infty}\frac{y_n}{n^2} < \infty, ~ \text{then} ~ \lim_{n \to \infty} \frac{x_n}{n} ~ \text{exists}.
    \end{equation*}
\end{thm}
\begin{proof}
    For proof, see de Bruijn and Erd\H{o}s \cite{de1952some}.
\end{proof}

\section{Asymptotic Behavior of the Arrival Process} \label{sec:3}
In this section, we state and prove the convergence of the arrival process $\left\{\xi_n \, \middle| \, n=1,2,\ldots \right\}$ to a Bernoulli random variable $\xi$.
\begin{thm} \label{th:con_in_D}
Given $\left(a_i\right)_{i=0}^{\infty}$  and $\left\{\xi_n \, \middle| \, n=1,2,\ldots \right\}$ as defined in Subsection \ref{Self-exciting process},
 let $\xi \colon \Omega \to \left\{0,1\right\}$ be a random variable such that the distribution of $\xi$ is given by
\begin{equation*}
    \mathbb{P}\left(\xi=1\right)=\frac{a_0}{1-\sum_{i=1}^{\infty} a_i} \quad \text{and} \quad \mathbb{P}\left(\xi=0\right)=1-\frac{a_0}{1-\sum_{i=1}^{\infty} a_i}.
\end{equation*}
Then 
    \begin{equation*}
        \xi_n \to \xi, \qquad \text{as} \ n \to \infty ~ \text{in distribution}.
    \end{equation*}
\end{thm}
Before proving Theorem \ref{th:con_in_D}, we state and prove Lemma \ref{lem:m.i}, which will be needed in the proof of Theorem \ref{th:con_in_D}.

\begin{lem}\label{lem:m.i}
    For the arrival process $\left\{\xi_n \, \middle| \, n=1,2,\ldots \right\}$, 
    \begin{equation*}
        \mathbb{E}\left(\xi_{n+1}\right)=\mathbb{P}\left(\xi_{n+1}=1\right)=a_0\left(1+b_1+\ldots+b_n\right),
    \end{equation*}
    where $b_n=a_n+\sum_{i=1}^{n-1}b_{n-i}a_i$, with $b_1=a_1.$
\end{lem}
\begin{proof}
To prove the lemma, we proceed using the Principle of Mathematical Induction. To this end, note that
    \begin{align*}
        \mathbb{E}\left(\xi_2\right)=\mathbb{P}\left(\xi_2=1\right)&=\mathbb{P}\left(\xi_2=1,\xi_1=1\right)+\mathbb{P}\left(\xi_2=1,\xi_1=0\right) \\
        &=\mathbb{P}\left(\xi_2=1 \, \middle| \, \xi_1=1 \right)\mathbb{P}\left(\xi_1=1\right)+ \mathbb{P}\left(\xi_2=1 \, \middle| \, \xi_1=0 \right)\mathbb{P}\left(\xi_1=0\right)\\
        &=\left(a_0+a_1\right)a_0 + a_0\left(1-a_0\right)\\
        &=a_0(1+a_1),
    \end{align*}
    and similarly,
    \begin{equation*}
        \mathbb{E}\left(\xi_3\right)=\mathbb{P}\left(\xi_3=1\right)=a_0\left(1+a_1+a_{1}^2+a_2\right)=a_0\left(1+b_1+b_2\right).
    \end{equation*}
    Now, suppose that the claim is true up to $n$, that is
    \begin{equation}\label{eq:ind1}
        \mathbb{E}\left(\xi_{r}\right)=\mathbb{P}\left(\xi_{r}=1\right)=a_0\left(1+b_1+\ldots+b_{r-1}\right),
    \end{equation}
    where $b_{r-1}=a_{r-1}+\sum_{i=1}^{r-2}b_{r-1-i}a_i$, for $2 \leq r \leq n$, with $b_1=a_1$.\\ 
    We want to show that the claim is true for $n+1$ also. To this end,
    \begin{equation}\label{eq:ce}
        \mathbb{E}\left(\xi_{n+1}\right)=\mathbb{E}\left(\mathbb{E}\left(\xi_{n+1} \, \middle| \, \mathcal{F}_n\right)\right)=\mathbb{E}\left(a_0+\sum_{i=1}^{n}a_{n+1-i}\xi_i\right)=a_0+\sum_{i=1}^{n}\left[a_{n+1-i}\mathbb{E}\left(\xi_i\right)\right].
    \end{equation}
    From (\ref{eq:ind1}) and (\ref{eq:ce}), we get
    \begin{align}\label{eq:3.5}
        \mathbb{E}\left(\xi_{n+1}\right)&=a_0+\sum_{i=1}^{n}\left[a_{n+1-i}a_0\left(1+\sum_{j=1}^{i-1}b_j\right)\right] \notag\\
        &=a_0\left[1+\sum_{i=1}^n a_{n+1-i} + \sum_{i=1}^n \sum_{j=1}^{i-1}a_{n+1-i}b_j \right] \notag \\
        & =a_0\left[1+\sum_{i=1}^{n}a_i+\sum_{j=1}^{n-1}\sum_{i=1}^{n-j}b_ja_i\right].
    \end{align}
    Similarly, we have
    \begin{equation}\label{eq:3.6}
        \mathbb{E}\left(\xi_{n}\right)=a_0\left[1+\sum_{i=1}^{n-1}a_i+\sum_{j=1}^{n-2}\sum_{i=1}^{n-1-j}b_ja_i\right].
    \end{equation}
    Subtracting (\ref{eq:3.6}) from (\ref{eq:3.5}) we get
    \begin{equation*}
        \mathbb{E}\left(\xi_{n+1}\right)-\mathbb{E}\left(\xi_n \right)=a_0\left[a_n+\sum_{i=1}^{n-1}b_{n-i}a_{i}\right],
    \end{equation*}
    which implies (using (\ref{eq:ind1}) with $r=n$),
    \begin{align*}
        \mathbb{E}\left(\xi_{n+1}\right)&=a_0\left[a_n+\sum_{i=1}^{n-1}b_{n-i}a_{i}\right] + \mathbb{E}\left(\xi_n \right)\\
        & = a_0\left[a_n+\sum_{i=1}^{n-1}b_{n-i}a_{i}\right]+a_0 \left(1+b_1+\ldots+b_{n-1}\right)\\
        &=a_0\left[1+\sum_{i=1}^{n-1}b_i+a_n+\sum_{i=1}^{n-1}b_{n-i}a_{i}\right].
    \end{align*}
    Thus, we get
    \begin{equation*}\
        \mathbb{E}\left(\xi_{n+1}\right)=a_0\left(1+\sum_{i=1}^{n}b_i\right),
    \end{equation*}
    where, $b_n=a_n +\sum_{i=1}^{n-1}b_{n-i}a_{i}$.

    Hence, from the Principle of Mathematical Induction, Lemma \ref{lem:m.i} follows.
\end{proof}
\bigskip

Now, we give the proof for Theorem \ref{th:con_in_D}.
\begin{proof}[Proof of Theorem \ref{th:con_in_D}]
For each $n \in \mathbb{N}$, let $F_{\xi_n}$ be the distribution function of the random variable $\xi_n.$ That is
\begin{equation*}
    F_{\xi_{n}}(x)=
    \begin{cases}
        0, &\text{if} \quad x<0,\\
        \mathbb{P}\left(\xi_n=0\right), &\text{if} \quad 0 \leq x < 1, \\
        1, &\text{if} \quad  x \geq 1.
    \end{cases}
\end{equation*}
Also, let $F_{\xi}$ be the distribution function of the random variable $\xi$. That is
\begin{equation*}
    F_{\xi}(x)=
    \begin{cases}
        0, &\text{if} \quad x<0,\\
        1-\frac{a_0}{1-\sum_{i=1}^{\infty}a_i}, & \text{if} \quad 0 \leq x < 1, \\
        1, & \text{if} \quad  x \geq 1.
    \end{cases}
\end{equation*}
From Lemma \ref{lem:m.i}, we have
\begin{equation*}
    \mathbb{E}\left(\xi_{n+1}\right)>\mathbb{E}\left(\xi_n\right), \qquad \text{for each} ~ n \in \mathbb{N}.
\end{equation*}
Moreover, we have that
\begin{equation*}
    \left|\mathbb{E}\left(\xi_{n}\right)\right|\leq 1, \qquad \text{for each} ~ n \in \mathbb{N}.
\end{equation*}
Therefore, by the Monotone Convergence Theorem, we conclude that
\begin{equation} \label{eq:3.5.1}
    \lim_{n \to \infty} \mathbb{E}\left(\xi_{n}\right) ~ \text{exists}.
\end{equation}

\bigskip
The SLLN for the DTHP given in (\ref{con:2.3}) allows us to write that
\begin{equation*}
    \frac{\sum_{i=1}^{n}\xi_i}{n} \to \frac{a_0}{1-\sum_{i=1}^{n}a_i}, \qquad \text{as} ~ n \to \infty ~ \text{a.s.}.
\end{equation*}
Since 
\begin{equation*}
    \left|\frac{\sum_{i=1}^{n}\xi_i}{n}\right| \leq 1,
\end{equation*} 
by the Dominated Convergence Theorem, we have that
\begin{equation} \label{con:E}
    \frac{\sum_{i=1}^{n}\mathbb{E}\left(\xi_i\right)}{n} \to \frac{a_0}{1-\sum_{i=1}^{\infty}a_i}, \qquad \text{as} ~ n \to \infty ~ \text{a.s.}.
\end{equation}
From \eqref{eq:3.5.1} and (\ref{con:E}), we get that 
\begin{equation*}
    \mathbb{E}\left(\xi_n\right) \to \frac{a_0}{1-\sum_{i=1}^{\infty}a_i}, \qquad \text{as} ~ n \to \infty ~ \text{a.s.}.
\end{equation*}
That is
\begin{equation}\label{con:P}
    \mathbb{P}\left(\xi_n=1\right) \to \frac{a_0}{1-\sum_{i=1}^{\infty}a_i}, \qquad \text{as} ~ n \to \infty ~ \text{a.s.}.
\end{equation}
Since, $\mathbb{P}\left(\xi_{n}=0\right)=1-\mathbb{P}\left(\xi_{n}=1\right),$ from (\ref{con:P}) we get 
\begin{equation}\label{con:xi_n=0}
    \mathbb{P}\left(\xi_n=0\right) \to 1-\frac{a_0}{1-\sum_{i=1}^{\infty}a_i}, \qquad \text{as} ~ n \to \infty ~ \text{a.s.}.
\end{equation}
Thus, from  (\ref{con:xi_n=0}), we conclude that
\begin{equation*}
    F_{\xi_{n}}(x) \to F_{\xi}(x), \qquad \text{as} ~ n \to \infty ~ \text{for all} ~ x \in \mathbb{R}\setminus \left\{0,1\right\}.
\end{equation*}
\end{proof}

\section{Asymptotic Behavior of the DTHP} \label{sec:4}
 In this section, we first establish the existence of LDP for the DTHP, as detailed in Subsection \ref{subsec 4.1}. Following this, we obtain the convergence of the scaled logarithmic MGF of the random variables $H_n, n=1,2,\ldots,$ in Subsection \ref{subsec 4.2}. Finally, we derive some estimates for the resulting limit function in Subsection \ref{subsec 4.3}.

\subsection{Large Deviation Principle}\label{subsec 4.1}
 \begin{thm}\label{Th 4.1}
     The DTHP $\left\{H_n \, \middle| \, n=1,2,\ldots\right\}$ satisfies the LDP with a good rate function.
 \end{thm}

The main tool used in the proof to establish the LDP for the DTHP is the Bryc's Theorem. We first introduce a well-separating family of functions in Lemma \ref{lem 4.3}. Then, we prove the necessary conditions of the Bryc's Theorem in Lemma \ref{lem 4.1} and Lemma \ref{lem 4.2}, and finally conclude the existence of LDP for the DTHP.
 
\begin{lem}\label{lem 4.3}
    The collection of functions $\mathcal{G} = \left\{g \colon [0,1] \to \mathbb{R} \, \middle| \, g ~ \text{is continuous and concave} \right\}$ is a well-separating collection of functions.
\end{lem}
\begin{proof}
    To verify that $\mathcal{G}$ is well-separating, we must check the three conditions of Definition \ref{def 2.3}. We give a detailed proof in \ref{sec_app_a}.
\end{proof}

\begin{lem}\label{lem 4.1}
The probability distributions of $\frac{H_n}{n}$ are exponentially tight.
\end{lem}

\begin{proof}
We follow Definition \ref{def_ET} for the proof. To this end, we choose $\mathbb{P}_n \left(\cdot \right)= \mathbb{P}\left(\frac{H_n}{n} \in \cdot \right)$. The detailed proof is given in \ref{sec_app_b}.
\end{proof}

\begin{lem}\label{lem 4.2}
Let $g \colon [0,1] \to \mathbb{R}$ be a concave and continuous function. Then, for the DTHP, the following limit exists,
\begin{equation*}
    \Gamma_g = \lim_{n \to \infty} \frac{1}{n} \log \mathbb{E} \left( e^{n g\left( \frac{H_n}{n} \right)} \right).
\end{equation*}
\end{lem}
\begin{proof}
Since $[0,1]$ is compact, there exists a constant $c > 0$ such that $-c \leq g(x) \leq c$. Now, for any $m, n, l \in \mathbb{N}$, the triangle inequality gives
\begin{equation}\label{eq 25.1}
    \frac{\left| H_{n+m} - \left(H_n + \left(H_{n+m+l} - H_{n+l} \right) \right) \right|}{n+m}
    \leq \frac{\left| H_{n+l} - H_n \right| + \left| H_{n+m+l} - H_{n+m} \right|}{n+m}.
\end{equation}
Using the fact that $H_n = \sum_{i=1}^{n} \xi_i$ and $\left| \xi_i \right| \leq 1$ for all $i \in \mathbb{N}$, we obtain
\begin{equation}\label{eq 25}
    \frac{\left| H_{n+l} - H_n \right| + \left| H_{n+m+l} - H_{n+m} \right|}{n+m}
    = \frac{\left| \sum_{i=n+1}^{n+l} \xi_i \right| + \left| \sum_{i=n+m+1}^{n+m+l} \xi_i \right|}{n+m} \leq \frac{2l}{n+m}.
\end{equation}
It is well known that every concave and continuous function on $[0,1]$ is also Lipschitz continuous. Therefore, there exists a constant $L > 0$ such that
\begin{equation}\label{eq 26}
    \left| g(x) - g(y) \right| \leq L \left| x - y \right|, \quad \text{for all } x, y \in [0,1].
\end{equation}
From \eqref{eq 25.1}, \eqref{eq 25}, and \eqref{eq 26}, we obtain
\begin{equation*}
    g\left( \frac{H_{n+m}}{n+m} \right) - g\left( \frac{H_n + \left(H_{n+m+l} - H_{n+l}\right)}{n+m} \right) \geq \frac{-2lL}{n+m}.
\end{equation*}
Since $g$ is concave, we can further write
\begin{align}\label{eq 28}
    g\left( \frac{H_{n+m}}{n+m} \right)
    &\geq g\left( \frac{n}{n+m} \cdot \frac{H_n}{n} + \frac{m}{n+m} \cdot \frac{H_{n+m+l} - H_{n+l}}{m} \right) - \frac{2lL}{n+m} \notag\\
    &\geq \frac{n}{n+m} g\left( \frac{H_n}{n} \right) + \frac{m}{n+m} g\left( \frac{H_{n+m+l} - H_{n+l}}{m} \right) - \frac{2lL}{n+m}.
\end{align}
Multiplying both sides of \eqref{eq 28} by $n + m$, we get
\begin{equation}\label{eq 29}
    (n + m) g\left( \frac{H_{n+m}}{n+m} \right)
    \geq n g\left( \frac{H_n}{n} \right) + m g\left( \frac{H_{n+m+l} - H_{n+l}}{m} \right) - 2lL.
\end{equation}
Taking exponential on both sides of \eqref{eq 29}, we obtain
\begin{equation}\label{eq 30}
    e^{(n + m) g\left( \frac{H_{n+m}}{n+m} \right)}
    \geq e^{n g\left( \frac{H_n}{n} \right) + m g\left( \frac{H_{n+m+l} - H_{n+l}}{m} \right) - 2lL}.
\end{equation}
Taking expectations on both sides of \eqref{eq 30}, and using the monotonicity of expectation, we get
\begin{equation}\label{eq 31}
    \mathbb{E}\left( e^{(n + m) g\left( \frac{H_{n+m}}{n+m} \right)} \right)
    \geq \mathbb{E}\left( e^{n g\left( \frac{H_n}{n} \right) + m g\left( \frac{H_{n+m+l} - H_{n+l}}{m} \right) - 2lL} \right).
\end{equation}
Taking logarithms on both sides of \eqref{eq 31}, we obtain
\begin{equation*}
    \log \mathbb{E} \left( e^{(n + m) g\left( \frac{H_{n+m}}{n+m} \right)} \right)
    \geq \log \mathbb{E} \left( e^{n g\left( \frac{H_n}{n} \right) + m g\left( \frac{H_{n+m+l} - H_{n+l}}{m} \right) - 2lL} \right),
\end{equation*}
which can be rewritten as
\begin{align}\label{eq 32.1}
    -\log \mathbb{E} \left( e^{(n + m) g\left( \frac{H_{n+m}}{n+m} \right)} \right)
    &\leq -\log \mathbb{E} \left( e^{n g\left( \frac{H_n}{n} \right) + m g\left( \frac{H_{n+m+l} - H_{n+l}}{m} \right) - 2lL} \right) \notag\\
    &= -\log \mathbb{E} \left( e^{n g\left( \frac{H_n}{n} \right)} e^{m g\left( \frac{H_{n+m+l} - H_{n+l}}{m} \right)} \right) + 2lL.
\end{align}

\bigskip
We now proceed to compute the covariance term
\begin{equation*}
\operatorname{Cov}\left[e^{n g\left(\frac{H_n}{n}\right)}, e^{m g\left(\frac{H_{n+m+l}-H_{n+l}}{m}\right)}\right].
\end{equation*}
Applying the general form of Hoeffding's covariance identity (for details, see Newman \cite{newman1980normal}), we obtain
\begin{align}\label{eq 25.3}
     & \left|\operatorname{Cov}\left[e^{ng\left(\frac{H_n}{n}\right)}, e^{mg\left(\frac{H_{n+m+l}-H_{n+l}}{m}\right)}\right]\right| \notag\\
     = & \left|\iint_{[0,1]^2} f'_n(x)f'_m(x)\operatorname{Cov}\left[\mathds{1}_{(-\infty,x]}\left(\frac{H_n}{n}\right),\mathds{1}_{(\infty,y]}\left(\frac{H_{n+m+l}-H_{n+l}}{m}\right) \right] dx dy \right|, ~ \text{for all} ~ x,y \in [0,1],
\end{align}
where 
\begin{equation*}
f_n(x) = e^{n g(x)}, ~ \text{for each} ~ x \in [0,1] ~ \text{and} ~ n \in \mathbb{N}.
\end{equation*}
Since for each $n \in \mathbb{N}$, $f_n(x)$ is an exponential function, we have the following bounds for its derivatives
\begin{equation}\label{ineq_exp}
\left|f'_n(x)\right| \leq n e^{n c} L, ~ \text{and} ~ \left|f'_m(y)\right| \leq m e^{m c} L, ~ \text{for each} ~ x \in [0,1] ~ \text{and} ~ n,m \in \mathbb{N}.
\end{equation}
For proof of \eqref{ineq_exp}, see \ref{sec_app_C}. Now, \eqref{eq 25.3} simplifies to
\begin{align}\label{eq 25.4}
    & \left|\operatorname{Cov}\left[e^{ng\left(\frac{H_n}{n}\right)}, e^{mg\left(\frac{H_{n+m+l}-H_{n+l}}{m}\right)}\right]\right| \notag\\
      \leq & \, n m e^{nc} e^{mc} L^2 \left|\iint_{[0,1]^2}\operatorname{Cov}\left[\mathds{1}_{(\infty,x]}\left(\frac{H_n}{n}\right),\mathds{1}_{(\infty,y]}\left(\frac{H_{n+m+l}-H_{n+l}}{m}\right)\right] dx dy \right|, \qquad \text{for all} ~ x,y \in [0,1]\notag\\
     \leq & \, n m e^{(n+m)c} L^2\left|\operatorname{Cov}\left[\frac{H_n}{n},\frac{H_{n+m+l}-H_{n+l}}{m} \right]\right|.
\end{align}
Using the bilinearity of the covariance operator, we rewrite the right-hand side of \eqref{eq 25.4} as follows,
\begin{align}\label{eq 26.3}
    n m e^{(n+m)c}L^2 \left|\operatorname{Cov}\left[\frac{H_n}{n},\frac{H_{n+m+l}-H_{n+l}}{m} \right]\right| & = e^{(n+m)c}L^2 \left|\operatorname{Cov}\left[H_n,H_{n+m+l}-H_{n+l} \right] \right| \notag\\
    & = e^{(n+m)c}L^2 \left|\operatorname{Cov}\left[\sum_{i=1}^n \xi_i,\sum_{j=n+l+1}^{n+l+m} \xi_j \right] \right| \notag\\
    & = e^{(n+m)c}L^2\left|\sum_{i=1}^n\sum_{j=n+l+1}^{n+l+m}\operatorname{Cov}\left[\xi_i,\xi_{j} \right] \right|.
\end{align}
Applying the triangle inequality to the right-hand side of \eqref{eq 26.3}, we obtain
\begin{align*}
    e^{(n+m)c}L^2\left|\sum_{i=1}^n\sum_{j=n+l+1}^{n+l+m}\operatorname{Cov}\left[\xi_i,\xi_{j} \right] \right| & \leq e^{(n+m)c}L^2\sum_{i=1}^n\sum_{j=n+l+1}^{n+l+m}\left|\operatorname{Cov}\left[\xi_i,\xi_{j} \right]\right|\\
    & \leq nm e^{(n+m)c} L^2.
\end{align*}
Finally, we derive the inequality
\begin{equation*}
    \left|\operatorname{Cov}\left[e^{ng\left(\frac{H_n}{n}\right)}, e^{mg\left(\frac{H_{n+m+l}-H_{n+l}}{m}\right)}\right]\right| \leq nme^{(n+m)c}L^2.
\end{equation*}
This implies
\begin{equation}\label{eq 34}
     \mathbb{E}\left(e^{ng\left(\frac{H_n}{n}\right)}\right)\mathbb{E}\left(e^{mg\left(\frac{H_{n+m+l}-H_{n+l}}{m}\right)} \right) - \mathbb{E}\left(e^{ng\left(\frac{H_n}{n}\right)} e^{mg\left(\frac{H_{n+m+l}-H_{n+l}}{m}\right)}\right) \leq nme^{(n+m)c}L^2.
\end{equation}
Dividing both sides of \eqref{eq 34} by $\mathbb{E}\left(e^{n g\left(\frac{H_n}{n}\right)} e^{m g\left(\frac{H_{n+m+l}-H_{n+l}}{m}\right)}\right)$, we have
\begin{align*}
    \frac{\mathbb{E}\left(e^{ng\left(\frac{H_n}{n}\right)}\right)\mathbb{E}\left(e^{mg\left(\frac{H_{n+m+l}-H_{n+l}}{m}\right)} \right)}{\mathbb{E}\left(e^{ng\left(\frac{H_n}{n}\right)} e^{mg\left(\frac{H_{n+m+l}-H_{n+l}}{m}\right)}\right)} &\leq 1 + \frac{nme^{(n+m)c}L^2}{\mathbb{E}\left(e^{ng\left(\frac{H_n}{n}\right)} e^{mg\left(\frac{H_{n+m+l}-H_{n+l}}{m}\right)}\right)}\\
   & \leq 1 + nme^{(n+m)c}e^{(n+m)c}L^2.
\end{align*}
Consequently,
\begin{equation*}
    \frac{\mathbb{E}\left(e^{ng\left(\frac{H_n}{n}\right)} e^{mg\left(\frac{H_{n+m+l}-H_{n+l}}{m}\right)}\right)}{\mathbb{E}\left(e^{ng\left(\frac{H_n}{n}\right)}\right)\mathbb{E}\left(e^{mg\left(\frac{H_{n+m+l}-H_{n+l}}{m}\right)} \right)} \geq \frac{1}{1 + nme^{(n+m)2c}L^2}.
\end{equation*}
Taking the logarithm on both sides, we get
\begin{align*}
    & \log\mathbb{E}\left(e^{ng\left(\frac{H_n}{n}\right)} e^{mg\left(\frac{H_{n+m+l}-H_{n+l}}{m}\right)}\right)\\
     \geq & \log\mathbb{E}\left(e^{ng\left(\frac{H_n}{n}\right)}\right) + \log\mathbb{E}\left(e^{mg\left(\frac{H_{n+m+l}-H_{n+l}}{m}\right)} \right) + \log\frac{1}{1+nme^{(n+m)2c}L^2}.
\end{align*}
Adding and subtracting $\log\mathbb{E}\left(e^{m g\left(\frac{H_m}{m}\right)}\right)$ on the right-hand side, we obtain
\begin{align*}
    &\log\mathbb{E}\left(e^{ng\left(\frac{H_n}{n}\right)} e^{mg\left(\frac{H_{n+m+l}-H_{n+l}}{m}\right)}\right)\\
     \geq & \log\mathbb{E}\left(e^{ng\left(\frac{H_n}{n}\right)}\right) + \log\mathbb{E}\left(e^{mg\left(\frac{H_m}{m}\right)}\right) + \log\frac{1}{1+nme^{(n+m)2c}L^2} + \log\mathbb{E}\left(e^{mg\left(\frac{H_{n+m+l}-H_{n+l}}{m}\right)} \right)\\
    & - \log\mathbb{E}\left(e^{mg\left(\frac{H_m}{m}\right)}\right).
\end{align*}
Applying the properties of logarithm and using the bound $-c \leq g(x) \leq c$, we have
\begin{align}\label{eq 35}
    &\log\mathbb{E}\left(e^{ng\left(\frac{H_n}{n}\right)} e^{mg\left(\frac{H_{n+m+l}-H_{n+l}}{m}\right)}\right) \notag\\
     \geq & \log\mathbb{E}\left(e^{ng\left(\frac{H_n}{n}\right)}\right) + \log\mathbb{E}\left(e^{mg\left(\frac{H_m}{m}\right)}\right) + \log\frac{1}{1+nme^{(n+m)2c}L^2} + \log\frac{\mathbb{E}\left(e^{mg\left(\frac{H_{n+m+l}-H_{n+l}}{m}\right)} \right)}{\mathbb{E}\left(e^{mg\left(\frac{H_m}{m}\right)}\right)} \notag\\
     \geq & \log\mathbb{E}\left(e^{ng\left(\frac{H_n}{n}\right)}\right) + \log\mathbb{E}\left(e^{mg\left(\frac{H_m}{m}\right)}\right) + \log\frac{1}{1+nme^{(n+m)2c}L^2} + \log\frac{e^{-mc}}{e^{mc}} \notag\\
     = & \log\mathbb{E}\left(e^{ng\left(\frac{H_n}{n}\right)}\right) + \log\mathbb{E}\left(e^{mg\left(\frac{H_m}{m}\right)}\right) + \log\frac{e^{-mc}}{e^{mc}\left(1+nme^{(n+m)2c}L^2\right)}.
\end{align}
Thus, from \eqref{eq 35} we have
\begin{align}\label{eq 36}
    &- \log\mathbb{E}\left(e^{ng\left(\frac{H_n}{n}\right)} e^{mg\left(\frac{H_{n+m+l}-H_{n+l}}{m}\right)}\right) \notag\\
     \leq & - \log\mathbb{E}\left(e^{ng\left(\frac{H_n}{n}\right)}\right) - \log\mathbb{E}\left(e^{mg\left(\frac{H_m}{m}\right)}\right) + \log e^{2mc}\left(1+nme^{(n+m)2c}L^2\right) \notag\\
      = &  - \log\mathbb{E}\left(e^{ng\left(\frac{H_n}{n}\right)}\right) - \log\mathbb{E}\left(e^{mg\left(\frac{H_m}{m}\right)}\right) + 2mc + \log \left(1 + nme^{(n+m)2c}L^2\right).
\end{align}
It is evident that for all $m, n \in \mathbb{N}$,
\begin{equation*}
    m < m+n, \qquad \text{and} \qquad nme^{(n+m)2c} < e^n e^m e^{(n+m)2c}=e^{(n+m)(1+2c)}.
\end{equation*}
Therefore, there exist constants $\delta_1 > 0$ and $\delta_2 > 0$ such that
\begin{equation*}
    m \leq (m+n)^{1-\delta_1}, \qquad \text{and} \qquad nme^{(n+m)2c} \leq e^{(n+m)^{1-\delta_2}\left(1+2c\right)}, \qquad \text{for all} ~ m,n \in \mathbb{N}.
\end{equation*}
Substituting these bounds into \eqref{eq 36}, we obtain
\begin{align}\label{eq38}
      - \log\mathbb{E}\left(e^{ng\left(\frac{H_n}{n}\right)} e^{mg\left(\frac{H_{n+m+l}-H_{n+l}}{m}\right)}\right) 
    \leq & -\log\mathbb{E}\left(e^{ng\left(\frac{H_n}{n}\right)}\right) - \log\mathbb{E}\left(e^{mg\left(\frac{H_m}{m}\right)}\right) + 2(m+n)^{1-\delta_1}c \notag\\
    & + \log \left(1 + e^{(n+m)^{1-\delta_2}(1+2c)}L^2\right).
\end{align}
Consequently, combining \eqref{eq 32.1} with \eqref{eq38}, we derive
\begin{align} \label{eq 39}
     - \log\mathbb{E}\left(e^{\left(n+m\right)g\left(\frac{H_{n+m}}{n+m}\right)} \right) \leq & -\log\mathbb{E}\left(e^{ng\left(\frac{H_n}{n}\right)}\right) - \log\mathbb{E}\left(e^{mg\left(\frac{H_m}{m}\right)}\right) + 2(m+n)^{1-\delta_1}c \notag\\
    & + \log \left(1 + e^{(n+m)^{1-\delta_2}\left(1+2c\right)}L^2\right) + 2lL.
\end{align}
Let us denote 
\begin{equation*}
    v_n = -\log\mathbb{E}\left(e^{ng\left(\frac{H_n}{n}\right)}\right).
\end{equation*}
Then, from \eqref{eq 39}, we can write
\begin{equation*}
    v_{n+m} \leq v_n + v_m + \overline{v}_{n+m},
\end{equation*}
where 
\begin{equation*}
\overline{v}_{n} = 2n^{1-\delta_1}c + \log\left(1+e^{n^{1-\delta_2}\left(1+2c\right)}L^2 \right) + 2lL.
\end{equation*}
Observe that, for each $n \in \mathbb{N}$,
\begin{equation*}
    \overline{v}_n \geq 0,
    \quad\text{and}\quad
    \overline{v}_{n+1} \geq \overline{v}_n,
\end{equation*}
i.e., $\left(\overline{v}_n\right)_{n=1}^\infty$ is a non-decreasing sequence of non-negative real numbers. Thus, the sequence $\left(v_n\right)_{n=1}^\infty$ is nearly subadditive with error term $\left(\overline{v}_n\right)_{n=1}^\infty$. Furthermore,
\begin{align*}
    \sum_{n=1}^{\infty} \frac{\overline{v}_n}{n^2} & = \sum_{n=1}^{\infty} \frac{2n^{1-\delta_1}c + \log\left(1+e^{n^{1-\delta_2}\left(1+2c\right)}L^2 \right) + 2lL}{n^2}\\
    & = \sum_{n=1}^{\infty}\frac{2c}{n^{1+\delta_1}} + \sum_{n=1}^{\infty}\frac{\log\left(1+e^{n^{1-\delta_2}\left(1+2c\right)}L^2 \right)}{n^2} + \sum_{n=1}^{\infty}\frac{2lL}{n^2}\\
    & < \infty.
\end{align*}
Hence, by Theorem \ref{th:2.4}, it follows that
\begin{equation*}
    \lim_{n \to \infty} \frac{v_n}{n} = \lim_{n \to \infty} \frac{-\log\mathbb{E}\left(e^{ng\left(\frac{H_n}{n}\right)}\right)}{n} ~ \text{exists,}
\end{equation*}
or equivalently, the following limit exists,
\begin{equation*}
    \lim_{n \to \infty} \frac{\log\mathbb{E}\left(e^{ng\left(\frac{H_n}{n}\right)}\right)}{n} = \Gamma_g .
\end{equation*}
\end{proof}

\begin{proof}[Proof of Theorem \ref{Th 4.1}]
    Now that we have shown the necessary conditions to apply the Bryc's Theorem for the DTHP in Lemma \ref{lem 4.3}, Lemma \ref{lem 4.1} and Lemma \ref{lem 4.2}, we conclude (using Bryc's Theorem) that the DTHP $\left\{H_n \, \middle| \, n=1,2\ldots \right\}$ satisfies the LDP with a good rate function given by
\begin{equation*}
    R(x)= \sup_{g \in \mathcal{C}_b\left([0,1]\right)} \left\{g(x)-\Gamma_g \right\}.
\end{equation*}
That is, the following holds,
\begin{enumerate}
    \item for any closed set $F \subseteq [0,1]$,
    \begin{equation}\label{eq_ldp_28}
        \limsup_{n \to \infty} \frac{1}{n}\log\mathbb{P}\left(\frac{H_n}{n} \in F \right) \leq -\inf_{x \in F} R(x),
    \end{equation}
    \item  for any open set $G \subseteq [0,1]$,
    \begin{equation}
        \liminf_{n \to \infty} \frac{1}{n}\log\mathbb{P}\left(\frac{H_n}{n} \in G \right) \geq -\inf_{x \in G} R(x).
    \end{equation}
\end{enumerate}

\end{proof}

In this subsection, we theoretically established the existence of LDP for the DTHP $\left\{H_n \, \middle| \, n=1,2,\ldots\right\}$. In the next subsection, we establish the connection between the rate function $R(x)$ and the pointwise limit of the scaled logarithmic MGF of the random variables $H_n, n=1,2,\ldots$.

\subsection{Convergence of scaled logarithmic MGF}\label{subsec 4.2}
In this subsection, we study the convergence of the scaled logarithmic MGF of the random variables $H_n, n=1,2,\ldots$. 

The scaled logarithmic MGF of the random variable $H_n$ is defined as follows,
\begin{equation*} 
\Gamma_n(t)=\frac{1}{n}\log \mathbb{E}\left(e^{tH_n}\right), ~ \text{for all} ~ t \in \mathbb{R}.
\end{equation*}

\begin{thm}\label{Th 4.2}
    For each $t \in \mathbb{R}$, the scaled logarithmic MGF of the random variables $H_n, n=1,2,\ldots,$ converges pointwise, i.e., there exists a function $\Gamma(t)$ such that
    \begin{equation*}
        \lim_{n \to \infty}\frac{1}{n}\log \mathbb{E}\left(e^{tH_n}\right) = \Gamma(t).
    \end{equation*}
\end{thm}

Before proving the theorem, note that, given the DTHP in Subsection \ref{Self-exciting process}, we have
\begin{equation}\label{eq:4.1}
\mathbb{E}\left(e^{tH_n}\right) = \mathbb{E}\left( \mathbb{E}\left(e^{tH_{n-1} +{t\xi_n}} \, \middle| \, \mathcal{F}_{n-1} \right)\right) = \mathbb{E}\left(e^{tH_{n-1}}\mathbb{E}\left(e^{t\xi_n}\, \middle| \, \mathcal{F}_{n-1}\right)\right),
\end{equation}
and 
\begin{align}\label{eq:4.2}
\mathbb{E}\left(e^{t\xi_n}\, \middle| \, \mathcal{F}_{n-1}\right)&=\mathbb{P}\left(\xi_n=0\, \middle| \, \mathcal{F}_{n-1}\right)+e^t\mathbb{P}\left(\xi_n=1\, \middle| \, \mathcal{F}_{n-1}\right)
\nonumber\\
&=1+\left(e^t-1\right) \mathbb{P}\left(\xi_n=1\, \middle| \, \mathcal{F}_{n-1}\right).
\end{align}
From (\ref{eq:4.1}) and (\ref{eq:4.2}), we obtain that
\begin{align*}
\mathbb{E}\left(e^{tH_n}\right)& =\mathbb{E}\left(e^{tH_{n-1}}\left[1+\left(e^t-1\right) \mathbb{P}\left(\xi_n=1\, \middle| \, \mathcal{F}_{n-1}\right)\right]\right) \\
& = \mathbb{E}\left(e^{tH_{n-1}}\right)+\left(e^t-1\right)\mathbb{E}\left(e^{tH_{n-1}}\mathbb{P}\left(\xi_n=1\, \middle| \, \mathcal{F}_{n-1}\right)\right)\\
& =\mathbb{E}\left(e^{tH_{n-1}}\right)+\left(e^t-1\right)\mathbb{E}\left(e^{tH_{n-1}}\left[a_0+\sum_{i=1}^{n-1}a_{n-i}\xi_i\right]\right).
\end{align*}
Hence,
\begin{equation}\label{eq:4.3}
    \mathbb{E}\left(e^{tH_n}\right)=\mathbb{E}\left(e^{tH_{n-1}}\right)+\left(e^t-1\right)\mathbb{E}\left(e^{tH_{n-1}}\left[a_0+\sum_{i=1}^{n-1}a_{n-i}\xi_i\right]\right).
\end{equation}
\bigskip
Further, for any $n \in \mathbb{N}$, the moment generating function of $H_n$ can also be written as
\begin{align}\label{eq 16.1}
    &\mathbb{E}\left(e^{tH_{n}}\right)=c_{0,n}+c_{1,n} e^t +\ldots + c_{n,n} e^{nt},\\
\text{where} \quad
    & c_{r,n}=\sum_{\substack{\alpha_1+\ldots +\alpha_n=r\\ \alpha_i \in \left\{0,1\right\}}} \mathbb{P}\left(\xi_1=\alpha_1, \xi_2=\alpha_2, \ldots, \xi_n=\alpha_n\right),~ \sum_{r=0}^n c_{r,n}=1, ~ 0\leq r \leq n. \notag
\end{align}
Using the multiplicative rule of probability, we get
\begin{align}\label{eq 16}
c_{0,n}&=\mathbb{P}\left(\xi_1=0, \xi_2=0, \ldots, \xi_n=0\right) \nonumber\\
&=\mathbb{P}\left(\xi_1=0\right)\mathbb{P}\left(\xi_2=0\, \middle| \, \xi_1=0\right)\ldots \mathbb{P}\left(\xi_n=0\, \middle| \, \xi_1=0,\xi_2=0,\ldots,\xi_{n-1}=0\right) \nonumber\\
 &= \left(1-a_0\right)^n.
\end{align}
Similarly,
\begin{align*}
c_{n,n}=a_0\left(a_0+a_1\right)\ldots\left(a_0+\ldots+a_{n-1}\right).
\end{align*}

\bigskip
Now, we prove the convergence in the following two lemmas.
\begin{lem} \label{lem:t<0}
    For $t \leq 0, \Gamma_n(t) $ converges pointwise.
\end{lem}
\begin{proof}
    From (\ref{eq:4.3}), for $t \leq 0$, we get
\begin{equation*}
    \mathbb{E}\left(e^{tH_n}\right) \leq \mathbb{E}\left(e^{tH_{n-1}}\right), \quad \text{for all} ~ n = 2,3,\ldots,
\end{equation*}
and consequently
\begin{equation*}
    \mathbb{E}\left(e^{tH_n}\right) \leq \mathbb{E}\left(e^{tH_{n-1}}\right) \leq \mathbb{E}\left(e^{tH_{n-2}}\right) \leq \ldots \leq \mathbb{E}\left(e^{tH_{1}}\right) \leq 1.
\end{equation*}
Thus, for any $m,n \in \mathbb{N}$,
\begin{equation*}
    \mathbb{E}\left(e^{tH_{m+n}}\right) \leq \mathbb{E}\left(e^{tH_{m}}\right), \quad \text{and} \quad \mathbb{E}\left(e^{tH_{m+n}}\right) \leq \mathbb{E}\left(e^{tH_{n}}\right).
\end{equation*}
Taking logarithms on both sides of the above  inequalities, we get
\begin{equation*}
    \log\mathbb{E}\left(e^{tH_{m+n}}\right) \leq \log\mathbb{E}\left(e^{tH_{m}}\right), \quad \text{and} \quad \log\mathbb{E}\left(e^{tH_{m+n}}\right) \leq \log\mathbb{E}\left(e^{tH_{n}}\right),
\end{equation*}
and adding the above inequalities gives the following relation,
\begin{equation}\label{eq:4.4.1}
    \log\mathbb{E}\left(e^{tH_{m+n}}\right) \leq \log\mathbb{E}\left(e^{tH_{m}}\right) + \log\mathbb{E}\left(e^{tH_{n}}\right) + \left[- \log\mathbb{E}\left(e^{tH_{m+n}}\right)\right].
\end{equation}
Let for each $t \leq 0$,
\begin{equation*}
h_n(t)=\log\mathbb{E}\left(e^{tH_{n}}\right), \quad \text{and} \quad \overline{h}_n(t)=-\log\mathbb{E}\left(e^{tH_{n}}\right).
\end{equation*}
Note that for each $t\leq0$, $\left(\overline{h}_n(t)\right)_{n=1}^{\infty}$ is a non-negative, non-decreasing sequence of real numbers. Thus, from \eqref{eq:4.4.1} we conclude that for each $t \leq 0$, $\left(h_n(t)\right)_{n=1}^\infty$ is a nearly subadditive sequence with error term $\left(\overline{h}_n(t)\right)_{n=1}^\infty.$\\
From (\ref{eq 16}), for $t \leq 0,$ we have
\begin{equation*}
    c_{0,n} < \mathbb{E}\left(e^{tH_n}\right)  \leq 1,
\end{equation*}
which implies successively
\begin{equation*}
    1 \leq \frac{1}{\mathbb{E}\left(e^{tH_n}\right)}  < \frac{1}{\left(1-a_0\right)^n},
\end{equation*}
and
\begin{equation*}
    0 \leq -\log\mathbb{E}\left(e^{tH_n}\right) <  n \log \frac{1}{1-a_0}.
\end{equation*}
Then, there exists a $\delta_3 > 0$ such that
\begin{equation*}
    0 \leq -\log\mathbb{E}\left(e^{tH_n}\right) \leq n^{1-\delta_3} \log \frac{1}{1-a_0}.
\end{equation*}
This implies for each $t \leq 0$,
\begin{equation*}
    \sum_{n=1}^{\infty}\frac{-\log\mathbb{E}\left(e^{tH_n}\right)}{n^2} < \infty.
\end{equation*}
Thus, from Theorem \ref{th:2.4}, we conclude that for each $t \leq 0$,
\begin{equation*}
    \lim_{n \to \infty} \frac{\log\mathbb{E}\left(e^{tH_n}\right)}{n} ~ \text{exists.}
\end{equation*}
\end{proof}
\begin{lem}\label{lem:t>0}
    For $t > 0, \Gamma_n(t) $ converges pointwise.
\end{lem}

\begin{proof}
    From (\ref{eq:4.3}), for $t>0$, we have
\begin{equation*}
    \mathbb{E}\left(e^{tH_n}\right) \leq \mathbb{E}\left(e^{tH_{n-1}}\right)+\left(e^t-1\right)\mathbb{E}\left(e^{tH_{n-1}}\sum_{i=0}^{n-1}a_i\right)= \mathbb{E}\left(e^{tH_{n-1}}\right)\left[1+\left(e^t-1\right)\sum_{i=0}^{n-1}a_i\right].
\end{equation*}
Recursively, we get
\begin{align}\label{eq 18}
    \mathbb{E}\left(e^{tH_n}\right) & \leq \mathbb{E}\left(e^{tH_{1}}\right)\left[1+\left(e^t-1\right)\sum_{i=0}^{1}a_i\right]\cdots\left[1+\left(e^t-1\right)\sum_{i=0}^{n-1}a_i\right] \notag\\
    & < \left[1+\left(e^t-1\right)\sum_{i=0}^{n-1}a_i\right]^n \notag\\ 
    & < \left[1+\left(e^t-1\right)\sum_{i=0}^{\infty}a_i\right]^n.
\end{align}
Then, there exists a $\delta_4 > 0$ such that 
\begin{equation}\label{eq:4.4}
    \mathbb{E}\left(e^{tH_n}\right) \leq \left[1+\left(e^t-1\right)\sum_{i=0}^{\infty}a_i\right]^{n^{1-\delta_4}}.
\end{equation}
Again, from (\ref{eq:4.3}), for $t>0$, we have
\begin{equation*}
    \mathbb{E}\left(e^{tH_n}\right) \geq \mathbb{E}\left(e^{tH_{n-1}}\right)\left[1+ \left(e^t-1\right)a_0 \right]=\mathbb{E}\left(e^{tH_{n-1}}\right)\mathbb{E}\left(e^{tH_1}\right).
\end{equation*}
Continuing recursively, we get
\begin{equation} \label{eq:4.6}
    \mathbb{E}\left(e^{tH_n}\right) \geq \mathbb{E}\left(e^{tH_{1}}\right)^n.
\end{equation}
Further,
\begin{equation} \label{eq:4.5}
    \left[\mathbb{E}\left(e^{tH_1}\right)\right]^n= \left[1 + \left(e^t - 1\right)a_0 \right ]^n > \left[1 + \left(e^t - 1\right)a_0 \right]^{n^{1-\delta_4}}.
\end{equation}
From (\ref{eq:4.4}),  (\ref{eq:4.6}) and (\ref{eq:4.5}), we get
\begin{equation*}
    1 \leq \frac{\mathbb{E}\left(e^{tH_n}\right)}{\left[\mathbb{E}\left(e^{tH_1}\right)\right]^n} \leq \frac{\left[1+\left(e^t - 1 \right)\sum_{i=0}^{\infty}a_i\right]^{n^{1-\delta_4}}}{\left[1+\left(e^t - 1 \right)a_0\right]^{n^{1-\delta_4}}}.
\end{equation*}
Taking logarithms, we get
\begin{equation} \label{ineq:con}
    0 \leq \log\frac{\mathbb{E}\left(e^{tH_n}\right)}{\left[\mathbb{E}\left(e^{tH_1}\right)\right]^n} \leq n^{1-\delta_4}\log \frac{1+\left(e^t - 1 \right)\sum_{i=0}^{\infty}a_i}{1+\left(e^t - 1 \right)a_0}.
\end{equation}

\bigskip
In Seol \cite{seol2015limit}, it is shown that, for any $n \in \mathbb{N},$ the random variables $\xi_1,\ldots,\xi_n$ are associated. Hence, for any $m,n \in \mathbb{N}$,
\begin{equation*}
    \mathrm{Cov}\left[e^{t\left(\xi_1+\ldots+\xi_m\right)},e^{t\left(\xi_{m+1}+\ldots+\xi_{m+n}\right)}\right] \geq 0,
\end{equation*}
which implies
\begin{equation}\label{ineq:4.7}
    \mathbb{E}\left(e^{tH_{m+n}}\right) \geq \mathbb{E}\left(e^{tH_m}\right)\mathbb{E}\left(e^{t\left(\xi_{m+1}+\ldots+\xi_{m+n}\right)}\right).
\end{equation}
Similarly,
\begin{equation}\label{ineq:4.8}
    \mathbb{E}\left(e^{tH_{n+m}}\right) \geq \mathbb{E}\left(e^{tH_n}\right)\mathbb{E}\left(e^{t\left(\xi_{n+1}+\ldots+\xi_{n+m}\right)}\right).
\end{equation}
From Lemma \ref{lem:m.i}, for any $m, n \in \mathbb{N}$, with $m \leq n$, we have
\begin{equation*}
    \mathbb{P}\left(\xi_{n}=1\right) > \mathbb{P}\left(\xi_{m}=1\right).
\end{equation*}
This implies, for any $m, n \in \mathbb{N}$, with $m \leq n$, and $t>0,$ we have
\begin{align*}
    \mathbb{E}\left(e^{t\xi_n}\right)  &= \mathbb{P}\left(\xi_n=0\right) + e^t\mathbb{P}\left(\xi_n=1\right)\\
    &=1+ \left(e^t -1\right)\mathbb{P}\left(\xi_n=1\right) \\
    & > 1+ \left(e^t -1\right)\mathbb{P}\left(\xi_m=1\right) = \mathbb{E}\left(e^{t\xi_m}\right).
\end{align*}
Multiplying (\ref{ineq:4.7}) and (\ref{ineq:4.8}), and using the fact that the random variables $\xi_1,\ldots,\xi_n$ are associated, we get
\begin{align*}
    \mathbb{E}\left(e^{tH_{m+n}}\right)^2 & \geq \mathbb{E}\left(e^{tH_m}\right)\mathbb{E}\left(e^{tH_n}\right)\left[\mathbb{E}\left(e^{t\xi_{m+1}}\right)\cdots\mathbb{E}\left(e^{t\xi_{m+n}}\right)\right]\left[\mathbb{E}\left(e^{t\xi_{n+1}}\right)\cdots\mathbb{E}\left(e^{t\xi_{n+m}}\right)\right] \\ 
    & \geq \mathbb{E}\left(e^{tH_m}\right)\mathbb{E}\left(e^{tH_n}\right)\mathbb{E}\left(e^{t\xi_{m+1}}\right)^n\mathbb{E}\left(e^{t\xi_{n+1}}\right)^m \\
& \geq \mathbb{E}\left(e^{tH_m}\right)\mathbb{E}\left(e^{tH_n}\right)\mathbb{E}\left(e^{t\xi_{1}}\right)^{m+n}.
\end{align*}
Thus, we get the following relation
\begin{equation*}
    \mathbb{E}\left(e^{tH_{m+n}}\right) \geq \mathbb{E}\left(e^{tH_m}\right)\mathbb{E}\left(e^{tH_n}\right)\frac{\mathbb{E}\left(e^{t\xi_{1}}\right)^{m+n}}{\mathbb{E}\left(e^{tH_{m+n}}\right)}.
\end{equation*}
Taking logarithms on both sides, we get
\begin{equation*}
    \log \mathbb{E}\left(e^{tH_{m+n}}\right) \geq \log \mathbb{E}\left(e^{tH_m}\right) + \log\mathbb{E}\left(e^{tH_n}\right) + \log\frac{\mathbb{E}\left(e^{t\xi_{1}}\right)^{m+n}}{\mathbb{E}\left(e^{tH_{m+n}}\right)}.
\end{equation*}
That is
\begin{equation*}
    -\log \mathbb{E}\left(e^{tH_{m+n}}\right) \leq -\log \mathbb{E}\left(e^{tH_m}\right) -\log\mathbb{E}\left(e^{tH_n}\right) + \log\frac{\mathbb{E}\left(e^{tH_{m+n}}\right)}{\mathbb{E}\left(e^{t\xi_{1}}\right)^{m+n}}.
\end{equation*}

\bigskip
For each $t>0$, let 
\begin{equation*}
    u_n(t)=-\log\mathbb{E}\left(e^{tH_n}\right), \quad \text{and} \quad \overline{u}_n(t) = \log\frac{\mathbb{E}\left(e^{tH_{n}}\right)}{\mathbb{E}\left(e^{t\xi_{1}}\right)^{n}}.
\end{equation*}
Note that,
\begin{align*}
    \overline{u}_{n+1}(t) - \overline{u}_n(t) & = \log\frac{\mathbb{E}\left(e^{tH_{n+1}}\right)}{\mathbb{E}\left(e^{t\xi_{1}}\right)^{n+1}} - \log\frac{\mathbb{E}\left(e^{tH_{n}}\right)}{\mathbb{E}\left(e^{t\xi_{1}}\right)^{n}} \\
    & = \log\mathbb{E}\left(e^{tH_{n+1}}\right)-(n+1)\log\mathbb{E}\left(e^{t\xi_1}\right)-\log\mathbb{E}\left(e^{tH_n}\right)+n\log\mathbb{E}\left(e^{t\xi_1}\right)\\
    &= \log \frac{\mathbb{E}\left(e^{tH_{n+1}}\right)}{\mathbb{E}\left(e^{tH_{n}}\right)} - \log \mathbb{E}\left(e^{t\xi_1}\right) \\
    & = \log \frac{\mathbb{E}\left(e^{tH_{n+1}}\right)}{\mathbb{E}\left(e^{tH_{n}}\right)\mathbb{E}\left(e^{t\xi_1}\right)} \geq 0.
\end{align*}
Also (using \eqref{ineq:con})
\begin{equation*}
    \overline{u}_n(t) \geq 0, \qquad \text{and} \qquad \sum_{n=1}^\infty \frac{\overline{u}_n(t)}{n^2} < \infty.
\end{equation*}
Thus, for each $t>0$, $\left(u_n(t)\right)_{n=1}^\infty$ is a nearly subadditive sequence with error term $\left(\overline{u}_n(t)\right)_{n=1}^\infty.$ Hence, by Theorem \ref{th:2.4}, we conclude that for each $t>0$,
\begin{equation*}
    \lim_{n \to \infty}\frac{-\log \mathbb{E}\left(e^{tH_{n}}\right)}{n} ~ \text{exists},
\end{equation*}
and consequently 
\begin{equation*}
\quad \lim_{n \to \infty}\frac{\log \mathbb{E}\left(e^{tH_{n}}\right)}{n} ~ \text{exists}.
\end{equation*}
\end{proof}

\begin{proof}[Proof of Theorem \ref{Th 4.2}]
    From Lemma \ref{lem:t<0} and Lemma \ref{lem:t>0}, we conclude that there exists a function $\Gamma \colon \mathbb{R} \to \mathbb{R}$, such that
\begin{equation*}
    \lim_{n \to \infty}\frac{\log \mathbb{E}\left(e^{tH_{n}}\right)}{n} = \Gamma(t),
\end{equation*}
and the proof follows.
\end{proof}

\begin{rem}
    Now that we have established the convergence of the scaled logarithmic MGF of the random variables $H_n$, $n = 1, 2, \ldots$, we invoke part \ref{GE 1} of the Gärtner–Ellis Theorem and obtain the following upper bound for the LDP satisfied by the DTHP,
\begin{equation}\label{eq 41}
\limsup_{n\rightarrow \infty} \frac{1}{n} \log\mathbb{P}\left(\frac{H_n}{n}\in F \right) \leq -\inf_{x\in F} \Gamma^*(x), \qquad \text{for any closed set} ~ F \subseteq [0,1],    
\end{equation}
where $\Gamma^{*}$ is the Fenchel-Legendre transform of $\Gamma$.
\end{rem}

\begin{rem}\label{rem_LDP_FLT}
    From \eqref{eq_ldp_28} and \eqref{eq 41}, we conclude that the rate function $R(x)$, associated with the upper bound of the LDP for the DTHP $\left\{H_n \, \middle| \, n=1,2,\ldots\right\}$, is given by the Fenchel-Legendre transform $\Gamma^{*}(x)$.
\end{rem}

We next establish differentiability properties of the limit function $\Gamma$.

\begin{prop}
    The limit function $\Gamma$ is differentiable almost everywhere on $\mathbb{R}$.
\end{prop}

\begin{proof}
    For each $n \in \mathbb{N}$, the scaled logarithmic MGF $\frac{1}{n}\log \mathbb{E}\left(e^{tH_n}\right)$, is convex. Since pointwise limits of convex functions are convex, the limit
    \begin{equation*}
      \Gamma(t)=\lim_{n\to\infty}\frac{1}{n}\log \mathbb{E}\left(e^{tH_n}\right),
    \end{equation*}
    is convex on $\mathbb{R}$. A convex function on $\mathbb{R}$ is differentiable at all points except on a set of Lebesgue measure $0$. Hence $\Gamma$ is differentiable almost everywhere.
\end{proof}

Since the rate function $R(x)$ is related to the Fenchel-Legendre transform of the limit function $\Gamma(t)$ (as stated in Remark \ref{rem_LDP_FLT}), we now estimate $\Gamma(t)$ with some known functions.

\subsection{Bounds for $\Gamma(t)$} \label{subsec 4.3}
In this subsection, we estimate the limit function $\Gamma(t)$. The following theorem gives a bound for the limit function $\Gamma(t)$.
\begin{thm}\label{Th 4.3}
    The following estimate for $\Gamma(t)$ holds,
    \begin{enumerate}
    \item For $t \geq 0,$
    \begin{equation*}
        \log\left(1+\left(e^t - 1 \right)\frac{a_0}{1-\sum_{i=1}^\infty a_i}\right) \leq \Gamma(t) \leq \log\left(1 + \left(e^t - 1 \right)\sum_{i=0}^{\infty}a_i \right).
    \end{equation*}
    \item For $t<0,$
    \begin{equation*}
        \operatorname{max}\left\{\log\left(1+\left(e^t - 1 \right)\frac{a_0}{1-\sum_{i=1}^\infty a_i}\right), \log\left(1-a_0 \right) \right\} \leq \Gamma(t) \leq \log\left(1 + \left(e^t - 1 \right)a_0 \right).
    \end{equation*}
    \end{enumerate}
\end{thm}

\begin{proof}[Proof of Theorem \ref{Th 4.3}] The proof is divided into two parts. First we obtain the lower bound and then the upper bound.
\subsection*{Lower bound}
For any $t \in \mathbb{R}$, from \eqref{ineq:4.8}, we have
\begin{equation*}
    \mathbb{E}\left(e^{tH_{n+1}}\right) \geq \mathbb{E}\left(e^{tH_n}\right)\mathbb{E}\left(e^{t\xi_{n+1}}\right),
\end{equation*}
which implies,
\begin{equation}\label{eq 25.2}
    \liminf_{n \to \infty} \frac{\mathbb{E}\left(e^{tH_{n+1}}\right)}{\mathbb{E}\left(e^{tH_n}\right)} \geq \liminf_{n \to \infty}\mathbb{E}\left(e^{t\xi_{n+1}}\right) = \liminf_{n \to \infty} \left[1 + \left(e^t - 1\right)\mathbb{P}\left(\xi_{n+1} = 1\right) \right].
\end{equation}
However, in Theorem \ref{th:con_in_D}, we have shown that $\lim_{n \to \infty} \mathbb{P}\left(\xi_n = 1 \right) = \frac{a_0}{1-\sum_{i=1}^\infty a_i}$. Thus, from \eqref{eq 25.2} we can write
\begin{equation}\label{eq 26.2}
    \liminf_{n \to \infty} \frac{\mathbb{E}\left(e^{tH_{n+1}}\right)}{\mathbb{E}\left(e^{tH_n}\right)} \geq 1+\left(e^t - 1 \right)\frac{a_0}{1-\sum_{i=1}^\infty a_i}.
\end{equation}
On the other hand, note that for any $t \in \mathbb{R}$,
    \begin{equation*}
        \lim_{n \to \infty}\frac{\log \mathbb{E}\left(e^{tH_{n}}\right)}{n} = \lim_{n \to \infty} \log \left( \mathbb{E}\left(e^{tH_n} \right)^{\frac{1}{n}} \right),
    \end{equation*}
    which implies, for any $t \in \mathbb{R}, 
        \lim_{n \to \infty} \mathbb{E}\left(e^{tH_n} \right)^{\frac{1}{n}}$ also exists. From the theory of real analysis, we also have
\begin{equation}\label{eq 26.1}
    \liminf_{n \to \infty} \frac{\mathbb{E}\left(e^{tH_{n+1}} \right)}{\mathbb{E}\left(e^{tH_{n}} \right)} \leq \liminf_{n \to \infty} \mathbb{E}\left(e^{tH_n} \right)^{\frac{1}{n}} = \lim_{n \to \infty}\mathbb{E}\left(e^{tH_n} \right)^{\frac{1}{n}}, \qquad \text{for any} ~ t \in \mathbb{R}.
\end{equation}
Combining \eqref{eq 26.2} and \eqref{eq 26.1} we get the following relation,
\begin{equation*}
    \lim_{n \to \infty}\mathbb{E}\left(e^{tH_n} \right)^{\frac{1}{n}} \geq 1+\left(e^t - 1 \right)\frac{a_0}{1-\sum_{i=1}^\infty a_i}, \qquad \text{for any} ~ t \in \mathbb{R},
\end{equation*}
or equivalently,
\begin{equation}\label{eq 30.1}
    \Gamma(t) = \lim_{n \to \infty}\log\left(\mathbb{E}\left(e^{tH_n} \right)^{\frac{1}{n}}\right) \geq \log\left(1+\left(e^t - 1 \right)\frac{a_0}{1-\sum_{i=1}^\infty a_i}\right), \qquad \text{for any} ~ t \in \mathbb{R}.
\end{equation}

\bigskip
We also have from \eqref{eq 16.1}, for any $t \in \mathbb{R}$,
\begin{equation*}
    \mathbb{E}\left(e^{tH_n} \right) \geq c_{0,n} = \left( 1-a_0 \right)^n,
\end{equation*}
or equivalently,
\begin{equation*}
    \left(\mathbb{E}\left(e^{tH_n} \right)\right)^{\frac{1}{n}} \geq 1 - a_0.
\end{equation*}
Then, taking the logarithm on both sides and letting $n \to \infty$, we get
\begin{equation}\label{eq 31.1}
    \Gamma(t) = \lim_{n \to \infty} \log\left(\left(\mathbb{E}\left(e^{tH_n} \right)\right)^{\frac{1}{n}} \right) \geq \log\left(1-a_0 \right), \qquad \text{for any} ~ t \in \mathbb{R}.
\end{equation}

\bigskip
Combining \eqref{eq 30.1} and \eqref{eq 31.1}, for any $t \in \mathbb{R}$,
\begin{equation*}
    \Gamma(t) \geq \operatorname{max}\left\{\log\left(1+\left(e^t - 1 \right)\frac{a_0}{1-\sum_{i=1}^\infty a_i}\right), \log\left(1-a_0 \right) \right\}.
\end{equation*}
\begin{rem}
    Note that, for $t \geq 0$,
    \begin{equation*}
        \operatorname{max}\left\{\log\left(1+\left(e^t - 1 \right)\frac{a_0}{1-\sum_{i=1}^\infty a_i}\right), \log\left(1-a_0 \right) \right\} = \log\left(1+\left(e^t - 1 \right)\frac{a_0}{1-\sum_{i=1}^\infty a_i}\right).
    \end{equation*}
\end{rem}

\subsection*{Upper bound}
For $t \geq 0,$ from \eqref{eq 18} we can write
\begin{equation*}
    \left(\mathbb{E}\left(e^{tH_n}\right)\right)^{\frac{1}{n}} \leq 1+\left(e^t-1\right)\sum_{i=0}^{\infty}a_i.
\end{equation*}
Taking logarithm and letting $n \to \infty$ on both sides, we get
\begin{equation*}
    \Gamma(t) = \lim_{n \to \infty} \log \left(\left(\mathbb{E}\left(e^{tH_n}\right)\right)^{\frac{1}{n}} \right) \leq \log\left(1+\left(e^t-1\right)\sum_{i=0}^{\infty}a_i \right), \qquad \text{for any} ~ t \geq 0.
\end{equation*}

\bigskip
For $t < 0$, from \eqref{eq:4.3} we can write
\begin{align*}
    \mathbb{E}(e^{tH_n})&=\mathbb{E}\left(e^{tH_{n-1}}\right)+\left(e^t-1\right)\mathbb{E}\left(e^{tH_{n-1}}\left[a_0+\sum_{i=1}^{n-1}a_{n-i}\xi_i\right]\right) \\
    & \leq \mathbb{E}\left(e^{tH_{n-1}}\right) + \left(e^t -1 \right)\mathbb{E}\left(e^{tH_{n-1}}a_0 \right) = \mathbb{E}\left(e^{tH_{n-1}}\right)\left[1+\left(e^t - 1\right)a_0 \right].
\end{align*}
Continuing recursively,
\begin{equation*}
    \mathbb{E}\left(e^{tH_n} \right) \leq \left[1 + \left(e^t - 1\right)a_0 \right]^n,
\end{equation*}
or equivalently
\begin{equation*}
    \left(\mathbb{E}\left(e^{tH_n} \right)\right)^{\frac{1}{n}} \leq 1 + \left(e^t - 1\right)a_0.
\end{equation*}
Taking logarithm and letting $n \to \infty$ on both sides, we get
\begin{equation*}
    \Gamma(t) = \lim_{n \to \infty} \log \left(\left(\mathbb{E}\left(e^{tH_n} \right)\right)^{\frac{1}{n}} \right) \leq \log\left(1 + \left(e^t - 1\right)a_0\right), \qquad \text{for any} ~ t < 0.
\end{equation*}
\end{proof}

\begin{rem}
    Note that in this paper, we have improved the bounds of the limit function $\Gamma(t)$ compared to the bounds found in Sarma and Selvamuthu \cite{sarma2024study}.
\end{rem}

Equivalently, the bounds in Theorem \ref{Th 4.3} can be expressed as follows. Define $L,U:\mathbb{R}\to\mathbb{R}$ by
\begin{equation*}
    L(t)=
    \begin{cases}
      \log\left(1+\left(e^{t}-1\right)\dfrac{a_0}{1-\sum_{i=1}^\infty a_i}\right), & t\ge 0,\\
      \max\left\{\log\left(1+\left(e^{t}-1\right)\dfrac{a_0}{1-\sum_{i=1}^\infty a_i}\right),\, \log\left(1-a_0\right)\right\}, & t<0,
    \end{cases}
\end{equation*}
and
\begin{equation*}
    U(t)=
    \begin{cases}
      \log\mathbb{E}\left(1+\left(e^{t}-1\right)\sum_{i=0}^{\infty} a_i\right), & t\ge 0,\\
      \log\mathbb{E}\left(1+\left(e^{t}-1\right)a_0\right), & t<0.
    \end{cases}
\end{equation*}
Then, for all $t\in\mathbb{R}$,
\begin{equation*}
    L(t)\le \Gamma(t)\le U(t).
\end{equation*}

\begin{rem}
    By the order-reversing property of the Fenchel–Legendre transform,
\begin{equation*}
    U^{*}(x)\le \Gamma^{*}(x)\le L^{*}(x), \qquad x\in\mathbb{R},
\end{equation*}
where $L^{*}$ and $U^{*}$ denote the Fenchel–Legendre transforms of $L$ and $U$, respectively.
\end{rem}

For a graphical illustration of the limit function $\Gamma$, consider the exciting function,
\begin{equation}\label{eq_values}
    a_0=0.2,\qquad a_i=0.3 \times 0.5^{\,i-1},\ \ i=1,2,\ldots.
\end{equation}
In Figure \ref{fig:bounds}(a) and Figure \ref{fig:bounds}(b), the feasible region (shaded in yellow) in which the limit function $\Gamma:\mathbb{R}\to\mathbb{R}$ lies for the choice \eqref{eq_values} is shown. Figure \ref{fig:bounds}(a) magnifies a neighborhood of the origin, whereas Figure \ref{fig:bounds}(b) provides a wider picture. The upper bound
\begin{equation*}
    U(t)=
    \begin{cases}
      \log\mathbb{E}\left(1+0.8\left(e^{t}-1\right)\right), & t\ge 0,\\
      \log\mathbb{E}\left(1+0.2\left(e^{t}-1\right)\right), & t<0,
    \end{cases}
\end{equation*}
is plotted in red, and the lower bound
\begin{equation*}
    L(t)=
    \begin{cases}
      \log\left(1+0.5\left(e^{t}-1\right)\right), & t\ge 0,\\
      \max\left\{\log\left(1+0.5\left(e^{t}-1\right)\right),\, \log(0.8)\right\}, & t<0,
    \end{cases}
\end{equation*}
is plotted in blue.

\begin{figure}[htbp]
\centering
\subfloat[Near origin]{%
  \includegraphics[width=0.49\textwidth]{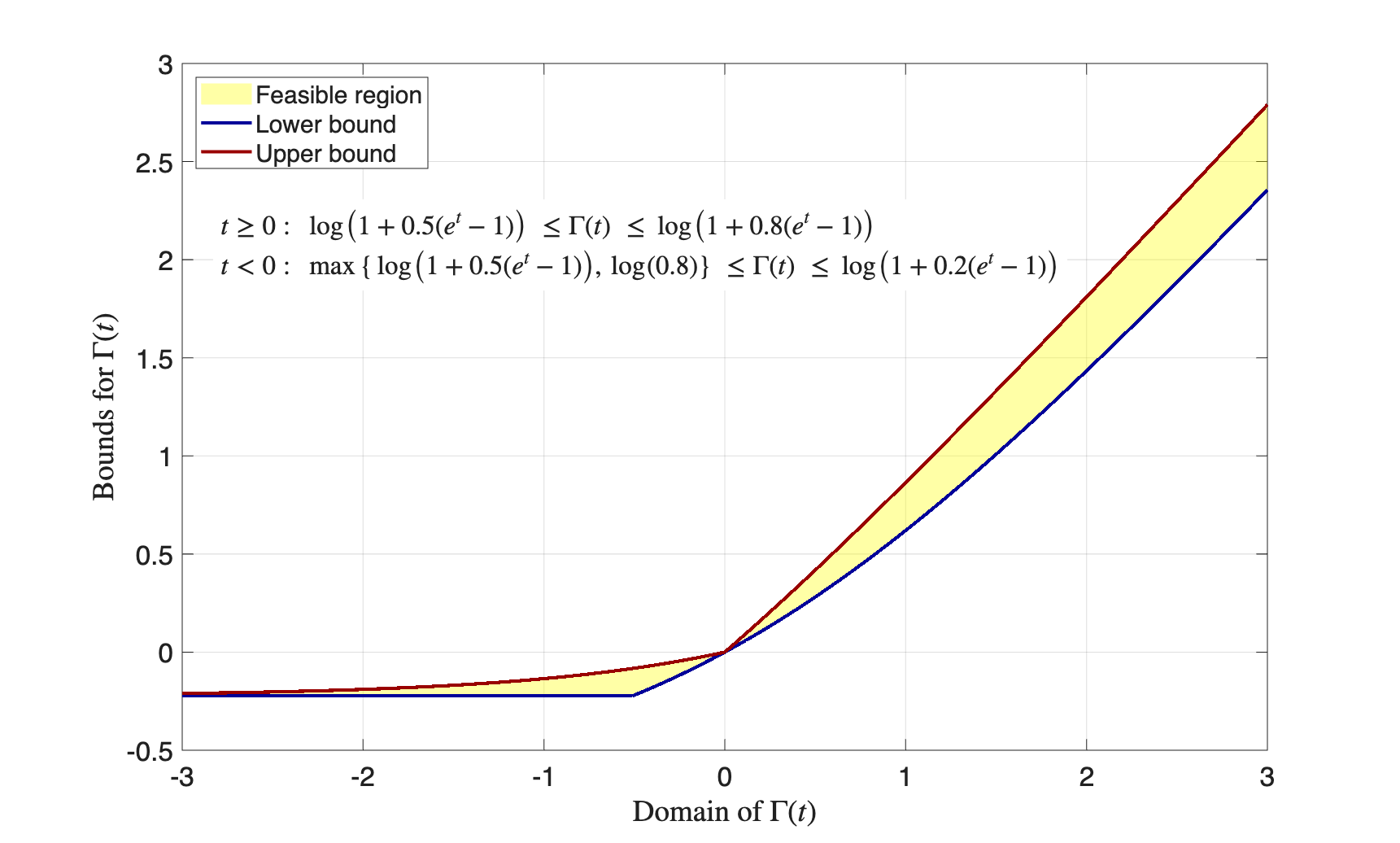}%
  \label{fig1}%
}\hfill
\subfloat[Wider view]{%
  \includegraphics[width=0.49\textwidth]{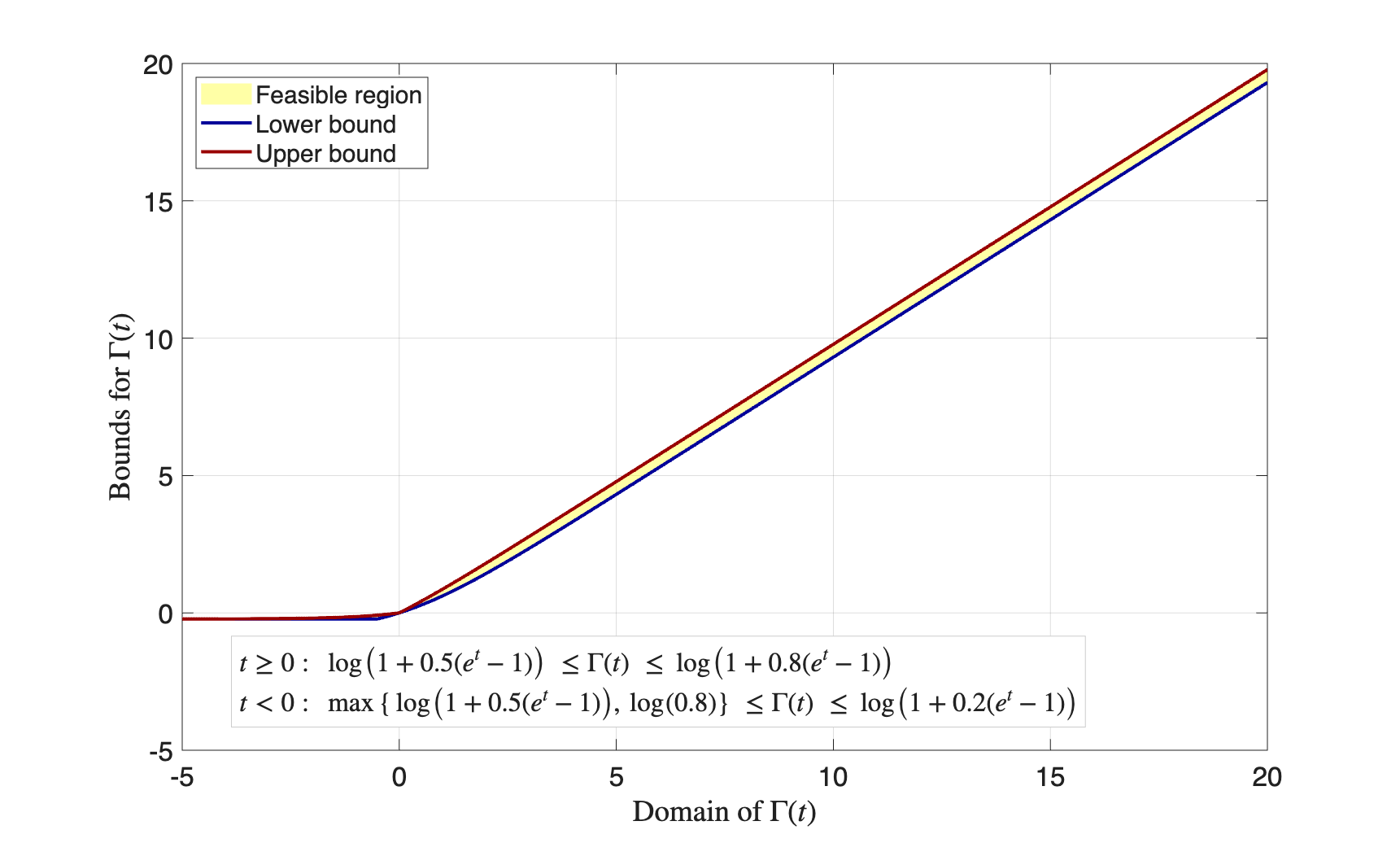}%
  \label{fig3}
}
\caption{Feasible region for the limit function $\Gamma(t)$ bounded by $L(t)$ (in blue) and $U(t)$ (in red).}
\label{fig:bounds}
\end{figure}

\section{Application in Finance}\label{Sec 5}
In Section \ref{sec:3} and Section \ref{sec:4}, we have proven the limiting behavior of the arrival process $\left\{\xi_n \, \middle| \, n=1,2,\ldots\right\}$, and the DTHP $\left\{H_n \, \middle| \, n=1,2,\ldots\right\}$. We can use them to construct various stochastic processes that are useful in real-life scenarios and study their limiting nature. In this section, we give an illustration that uses this discrete-time Hawkes model. Taking motivation from the model described in Stabile and Torrisi \cite{stabile2010risk}, we create a discrete-time model that describes the surplus amount left with an insurance company.

\subsection{Surplus process with discrete-time Hawkes claim arrivals}

Consider a discrete-time surplus process $\left\{U_n \, \middle| \, n=1,2,\ldots\right\}$, defined by

\begin{equation}\label{eq 40}
U_n = u + n p  - \sum_{i=1}^{n} \xi_i, \quad n =1,2,\ldots,
\end{equation}
where
\begin{itemize}
    \item $u \in (0,1)$ is the initial surplus,
    \item \(p \in (0,1)\) is the constant premium received in each discrete time step,
    \item $\left\{\xi_n \, \middle| \, n=1,2,\ldots\right\}$ is the arrival process (defined in Subsection \ref{sec model}) indicating claim occurrences. The event $\left\{\xi_n=1\right\}$ indicates that the $n$-th customer files a claim of amount $\$ 1$ (for simplicity), whereas $\left\{\xi_n =0\right\}$ indicates no claim.
\end{itemize}

A natural question that one may ask about this model is, what is the premium the insurance company should charge the customer so that it does not lose all its surplus amount? The answer to this question is given in the following proposition.

\begin{prop}\label{prop_5.1}
    For the discrete-time surplus process $\left\{U_n \, \middle| \, n=1,2,\ldots\right\}$, defined in \eqref{eq 40}, the minimum value of the premium that the insurance company should charge to remain in profit in the long run is given by
    \begin{equation*}
        p > \frac{a_0}{1 - \sum_{i=1}^\infty a_i}.
    \end{equation*}
\end{prop}
\begin{proof}
In order for the insurance company to remain in profit in the long run, the following condition should hold, 
\begin{equation*}
        \lim_{n \to \infty} \frac{\mathbb{E}\left(U_n\right)}{n} > 0.
\end{equation*}
From \eqref{eq 40}, we get (using the fact that expectation is a linear operator)
\begin{equation*}
    \mathbb{E}\left(U_n\right) = u + np - \sum_{i=1}^n \mathbb{E}\left(\xi_i \right).
\end{equation*}
Dividing both sides by $n$ and letting $n \to \infty$, we have
\begin{equation*}
    \lim_{n \to \infty} \frac{\mathbb{E}\left(U_n \right)}{n} = \lim_{n \to \infty} \frac{u}{n} + p - \lim_{n \to \infty} \frac{\sum_{i=1}^n \mathbb{E}\left(\xi_i\right)}{n}.
\end{equation*}
Finally, using \eqref{con:E} we get
\begin{equation*}
    \lim_{n \to \infty} \frac{\mathbb{E}\left(U_n\right)}{n} = p -  \frac{a_0}{1-\sum_{i=1}^\infty a_i}.
\end{equation*}
Thus to get $\lim_{n \to \infty}\frac{\mathbb{E}\left(U_n\right)}{n} > 0$, we have to choose $p$ such that 
\begin{equation*}
    p > \frac{a_0}{1 - \sum_{i=1}^\infty a_i}.
\end{equation*}
Hence, the insurance company should charge a slightly higher premium than $\frac{a_0}{1 - \sum_{i=1}^\infty a_i}$ to make a profit in the long run.
\end{proof}

But there is still some positive probability for the company to go bankrupt. We now answer the question that even if the company charges a premium greater than $\frac{a_0}{1 - \sum_{i=1}^\infty a_i}$, what is the probability that it can still go bankrupt? This question can be answered using the LDP for the DTHP proved in Subsection \ref{subsec 4.1}. In order for the company to go bankrupt in some finite time $n$, the surplus amount should be negative, i.e., $U_n < 0$. This is equivalent to saying that 
\begin{equation*}
    u + np - H_n <0,
\end{equation*}
that is
\begin{equation*}
    H_n > u + np, \qquad \text{or} \qquad \frac{H_n}{n} > \frac{u}{n} + p.
\end{equation*}
Using LDP for the DTHP, we can approximate
\begin{equation*}
    \mathbb{P}\left(\frac{H_n}{n} \in \left(\frac{u}{n} +p, 1 \right) \right) \approx e^{-n \inf_{x \in \left(\frac{u}{n} + p, 1 \right)}R(x)}.
\end{equation*}

Thus, even if $p > \frac{a_0}{1 - \sum_{i=1}^\infty a_i}$, there is a small probability that the insurance company can go bankrupt at some time $n$, and that probability is approximately $e^{-n \inf_{x \in \left(\frac{u}{n} + p, 1 \right)}R(x)}$ (this probability is also called the ruin probability).

\subsection{Numerical Simulations}

In this subsection, we present simulations of the DTHP $\left\{H_n \, \middle| \, n=1,2,\ldots\right\}$ and the corresponding surplus process $\left\{U_n \, \middle| \, n=1,2,\ldots\right\}$. For this purpose, we assume the exciting function of the DTHP to be given by \eqref{eq_values}. To simulate the surplus process, we set the initial surplus  $u=0.6 \in (0,1)$, and the constant premium $p=0.6 \in (0,1)$. Figure \ref{sample_path_DTHP_SP}(a) shows a single realization (sample path) of the DTHP $\left\{H_n \, \middle| \, n=1,2,\ldots\right\}$, where the red dots indicate the time points of claim arrivals. The corresponding realization of the surplus process $\left\{U_n \, \middle| \, n=1,2,\ldots\right\}$ is shown in Figure \ref{sample_path_DTHP_SP}(b).

\begin{figure}[htbp]
\centering
\subfloat[DTHP]{%
  \includegraphics[width=0.5\textwidth]{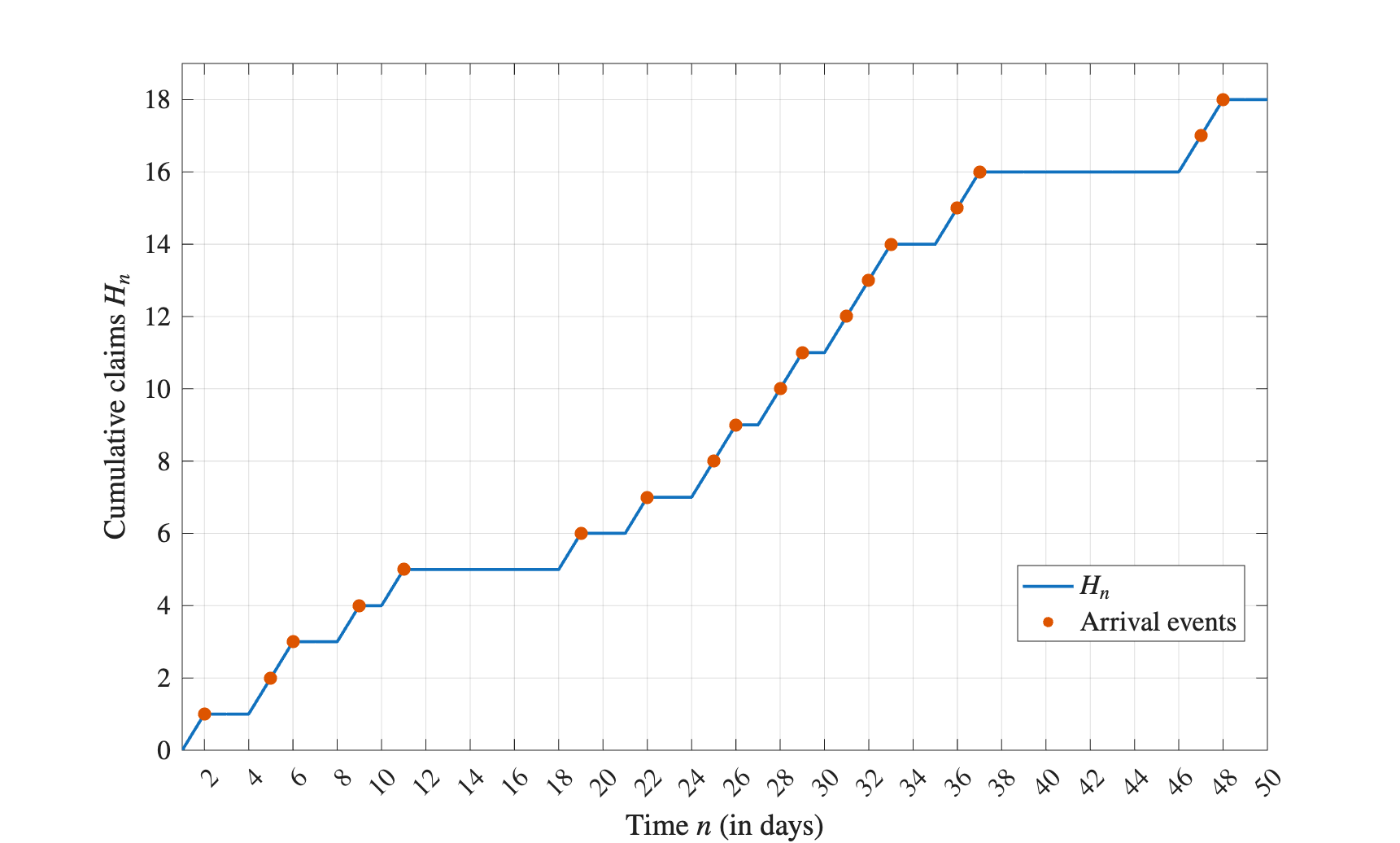}%
  \label{fig_DTHP}%
}\hfill
\subfloat[Surplus process]{%
  \includegraphics[width=0.5\textwidth]{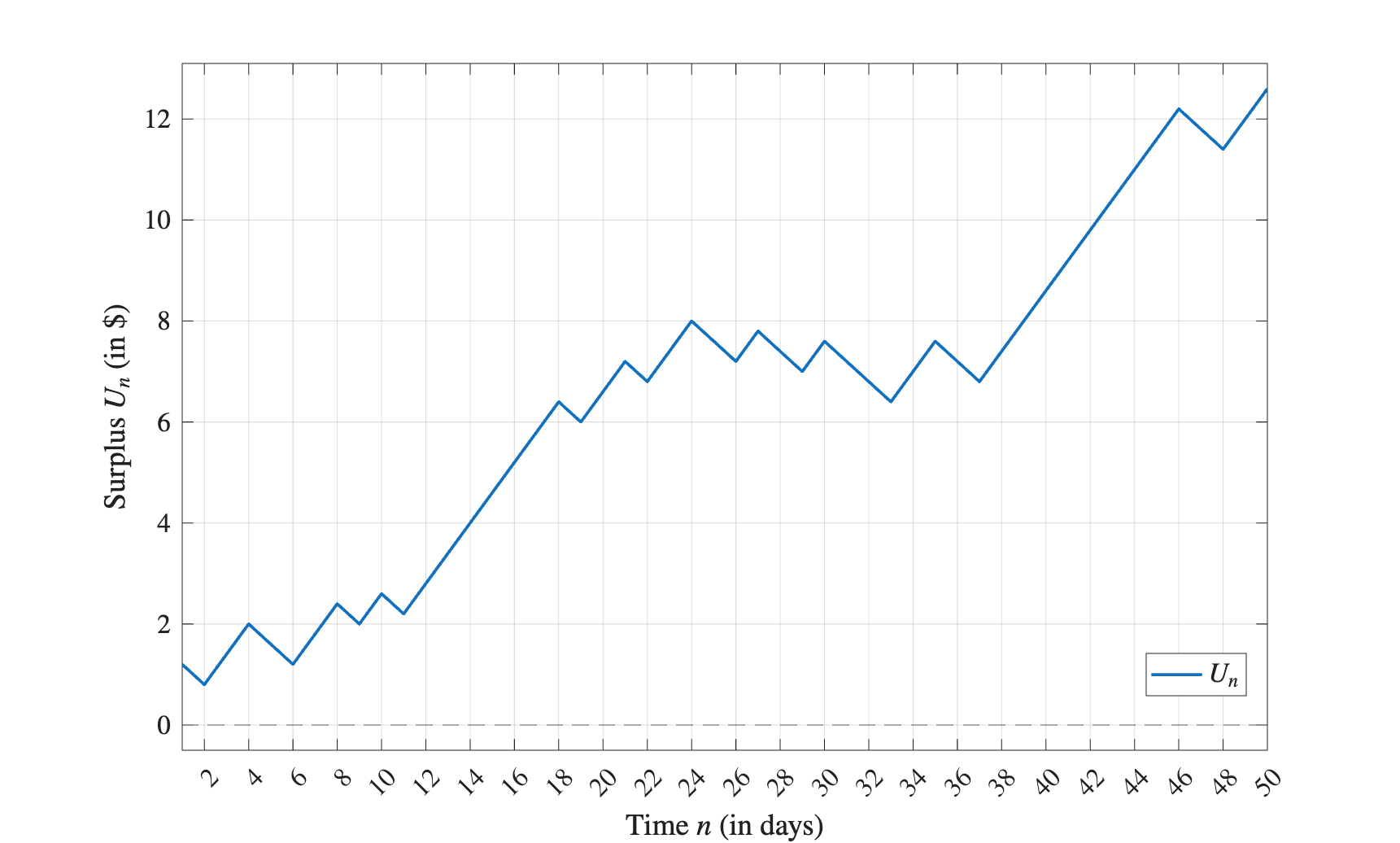}%
  \label{fig_Surplus}%
}
\caption{Sample paths of the DTHP $\left\{H_n \, \middle| \, n=1,2,\ldots\right\}$ and the corresponding surplus process $\left\{U_n \, \middle| \, n=1,2,\ldots\right\}$.}
\label{sample_path_DTHP_SP}
\end{figure}

We now perform Monte Carlo simulations of the surplus process $\left\{U_n \, \middle| \, n=1,2,\ldots\right\}$, where we generate $100{,}000$ realizations (sample paths) of the surplus process $\left\{U_n \, \middle| \, n=1,2,\ldots\right\}$. For each time step $n=1,2,\ldots$, we calculate the $5$th percentile, the $95$th percentile, and the mean of all the $100,000$ realizations (sample paths) and plot them in Figure \ref{fig_MC}. The region between the 5th percentile curve and the 95th percentile curve is shaded in gray, and the mean curve of all the $100,000$ realizations is shown as a solid black line. This means that about 90\% of the simulated realizations lie in the gray shaded region, about 5\% are below it, and 5\% are above it. 

In Figure \ref{fig_MC}(a), we set the initial surplus $u=0.6$, the constant premium $p=0.4$, and the exciting function as mentioned in \eqref{eq_values}. Since the premium $p$ is strictly less than $\frac{a_0}{1-\sum_{i=1}^\infty a_i}=0.5,$ the mean curve of the surplus process exhibits a downward trend, attaining negative values and decreasing steadily. This indicates an eventual ruin of the insurance company. In contrast, Figure \ref{fig_MC}(b) uses $u=0.6$, and $p=0.6$, with the same exciting function as mentioned in \eqref{eq_values}. Here, since the premium exceeds $\frac{a_0}{1-\sum_{i=1}^\infty a_i}=0.5,$ the mean curve of the surplus process exhibits an upward trend, increasing steadily over time. This indicates long-term profitability for the insurance company.

Note that in Proposition \ref{prop_5.1}, we have obtained that if 
\begin{equation*}p > \frac{a_0}{1 - \sum_{i=1}^\infty a_i},
\end{equation*}
the insurance company will be in profit in the long run. That is 
\begin{equation*}
\lim_{n \to \infty} \frac{\mathbb{E}\left(U_n\right)}{n} > 0, ~ \text{or} ~ \lim_{n \to \infty} \mathbb{E}\left(U_n\right) > 0,
\end{equation*}
and this is exactly what we obtained from the Monte Carlo simulation results. Thus, the Monte Carlo simulation results verify Proposition \ref{prop_5.1}, and also enhance the impact of this section by providing a concrete numerical illustration of the DTHP model.

\begin{figure}[htbp]
\centering
\subfloat[premium $p=0.4<\frac{a_0}{1-\sum_{i=1}^\infty a_i}=0.5$]{%
  \includegraphics[width=0.5\textwidth]{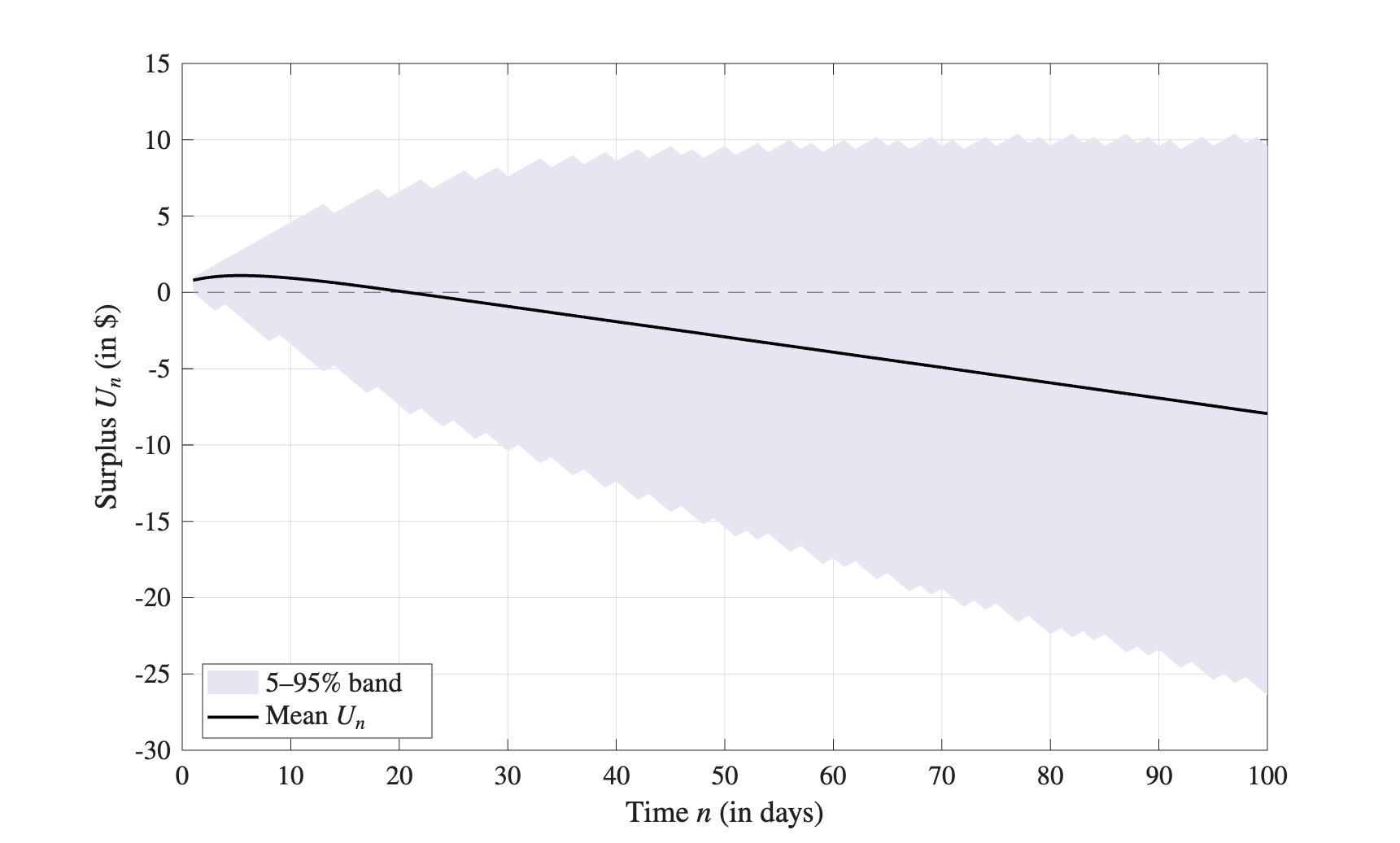}%
  \label{fig_MC_p=0.4}%
}\hfill
\subfloat[premium $p=0.6>\frac{a_0}{1-\sum_{i=1}^\infty a_i}=0.5$]{%
  \includegraphics[width=0.5\textwidth]{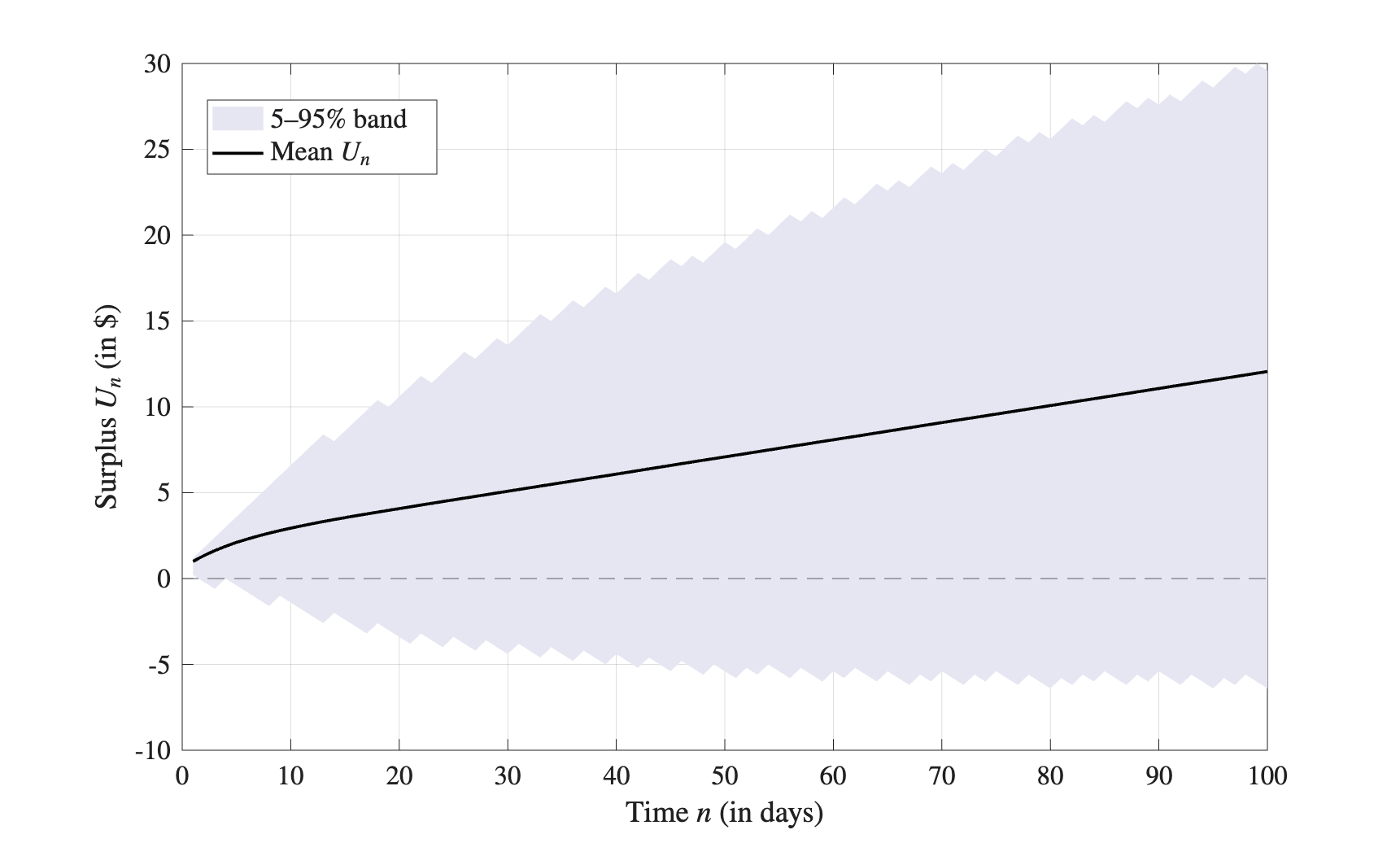}%
  \label{fig_MC_p=0.6}%
}
\caption{Monte Carlo simulations of the surplus process $\left\{U_n \, \middle| \, n=1,2,\ldots\right\}$ with $100,000$ sample paths.}
\label{fig_MC}
\end{figure}

\section{Conclusion and Future Work}\label{Sec 6.1}
In this paper, we introduced an arrival process $\left\{\xi_n \, \middle| \, n = 1, 2, \ldots\right\}$, and the corresponding point process (DTHP) $\left\{H_n \, \middle| \, n = 1, 2, \ldots\right\}$, that counts the number of arrivals up to time $n$. We established the weak convergence of the arrival process to a Bernoulli random variable, thereby characterizing its asymptotic behavior. Furthermore, we derived the LDP for the DTHP and identified its rate function. In addition, we proved that the scaled logarithmic MGF of the random variables $H_n, n=1,2,\ldots$, converges to a limiting function $\Gamma(t)$, obtained its relation with the rate function, and estimated the upper and lower bounds for $\Gamma(t)$ to better understand the asymptotic nature of the DTHP. Finally, to illustrate the practical relevance of the results obtained in this paper, we concluded the study with an example demonstrating a real-world application of the DTHP model.

In future, we intend to find an explicit expression of the rate function $R(x)$, which would estimate the probabilities of rare events more accurately. Another interesting direction could be to incorporate mutual excitation and inhibition, along with self-excitation, into the DTHP model and investigate its asymptotic behavior.

\appendix

\section{Proof of Lemma \ref{lem 4.3}}\label{sec_app_a}
Any constant function $g \colon [0,1] \to \mathbb{R}$ is continuous and concave. Thus, $g \in \mathcal{G}$. Now, let $g_1,g_2\in\mathcal{G}$ and define $g(x)=\min\left\{g_1(x),g_2(x)\right\}$. The function $g$ is continuous, being the pointwise minimum of continuous functions. Moreover, the pointwise minimum of two concave functions remains concave. Thus, $g\in\mathcal{G}$, which proves the closure. Lastly, let $x,y\in[0,1]$, with $x<y$, and let $a,b\in \mathbb{R}$ be arbitrary. Consider the linear function $g \in \mathcal{G}$, defined by
\begin{equation*}
g(t)=\frac{b-a}{y-x}(t-x)+a.
\end{equation*}
Clearly, $g(x)=a$ and $g(y)=b$. Thus, $\mathcal{G}$ separates the points of $[0,1]$. Therefore, all the conditions are satisfied and $\mathcal{G}$ is well-separating.

\section{Proof of Lemma \ref{lem 4.1}} \label{sec_app_b}
Let $\varepsilon<\infty$. Define $K_{\varepsilon} = [0,1] = K$ for each such $\varepsilon$. Since, for each $n \in \mathbb{N}$,
\begin{equation*}
    0 \leq \frac{H_n}{n} \leq 1, 
\end{equation*}
it follows that
\begin{equation*}
\left\{ \omega \in \Omega \, \middle| \, \frac{H_n}{n}(\omega) \in K^\complement \right\} = \phi.
\end{equation*}
This implies
\begin{equation*}
\limsup_{n \to \infty} \frac{1}{n} \log \mathbb{P}\left( \frac{H_n}{n} \in K^\complement \right) = -\infty < -\varepsilon,
\end{equation*}
and the proof follows.

\section{}\label{sec_app_C}
Recall the sequence of functions $f_n(x) = e^{n g(x)}$, for each $x \in [0,1]$, and $n \in \mathbb{N}$. Since the exponential function $e^x$ is differentiable everywhere, it is certainly differentiable on the interval $[-n c, n c]$ for all $n \in \mathbb{N}$. Applying the mean-value theorem, we know there exists points $p_n \in [-n c, n c]$ satisfying
\begin{equation*}
\left| \frac{e^t - e^s}{t - s} \right| = e^{p_n} \leq e^{n c}, \quad \text{for all } n \in \mathbb{N}, \text{ and } t,s \in [-n c, n c].
\end{equation*}
This immediately implies
\begin{equation*}
\left| e^t - e^s \right| \leq e^{n c} \left| t - s \right|, \quad \text{for all } n \in \mathbb{N}, \text{ and } t,s \in [-n c, n c].
\end{equation*}
Hence, we obtain the inequality
\begin{equation*}
\left| e^{n g(x)} - e^{n g(y)} \right| \leq e^{n c} \left| n g(x) - n g(y) \right| = n e^{n c} \left| g(x) - g(y) \right|.
\end{equation*}
Combining this result with the Lipschitz continuity condition given in \eqref{eq 26}, we further derive
\begin{equation*}
\left| e^{n g(x)} - e^{n g(y)} \right| \leq n e^{n c} L |x - y|, \quad \text{for all } x,y \in [0,1].
\end{equation*}
Thus, for each $n \in \mathbb{N}$, the function $f_n$ is Lipschitz continuous on $[0,1]$, and hence is almost everywhere differentiable. Moreover, wherever the derivative exists, we have the bound
\begin{equation*}
\left|f'_n(x)\right| \leq n e^{n c} L.
\end{equation*}

\end{document}